\documentclass[a4paper,11pt]{amsart}
\usepackage{amsfonts,amssymb,amsmath,abstract}
\usepackage{pb-diagram}
\usepackage[lmargin=1in,rmargin=1in,tmargin=1in,bmargin=1in]{geometry}
\usepackage{fancyhdr}
\usepackage{graphicx}

\usepackage[ps,all,arc,rotate]{xy}

\newtheorem{prop}{Proposition}[section]
\newtheorem{lemma}[prop]{Lemma}
\newtheorem{thm}[prop]{Theorem}
\newtheorem{cor}[prop]{Corollary}

\theoremstyle{definition}
\newtheorem{defn}[prop]{Definition}

\newtheorem{rmk}[prop]{Remark}
\newtheorem{ex}[prop]{Example}

\mathchardef\mhyphen="2D

\begin{document}

\title{Quotients of unstable subvarieties and moduli spaces of sheaves
of fixed Harder--Narasimhan type}

\author{Victoria Hoskins and Frances Kirwan}

\thanks{This work was supported by the Engineering and Physical Sciences
Research Council [grant number   GR/T01341/01].}
\maketitle

\renewcommand{\abstractnamefont}{\scshape}

\begin{abstract}

When a reductive group $G$ acts linearly on a complex projective scheme $X$ there is a stratification of $X$ into $G$-invariant locally closed subschemes, with an open stratum $X^{ss}$ formed by the semistable points in the sense of Mumford's geometric invariant theory which has a categorical quotient $X^{ss} \to X/\!/G$. In this article we describe a method for constructing quotients of the unstable strata. As an application, we construct moduli spaces of sheaves of fixed Harder--Narasimhan type with some extra data (an \lq $n$-rigidification') on a projective base.
\end{abstract}

\pagestyle{fancy}
\fancyhead{}
\fancyfoot{}
\fancyhead[CE]{\textsc{victoria hoskins and frances kirwan}}
\fancyhead[CO]{\textsc{\rightmark}}
\fancyhead[RO,LE]{\thepage}

\section{Introduction}\markright{introduction}

Let $X$ be a complex projective scheme and $G$ a complex reductive group acting
linearly on $X$ with respect to an ample line bundle. Mumford's geometric invariant theory (GIT) \cite{mumford} provides us with a projective scheme
$X/\!/G$ which is a categorical quotient of an open subscheme $X^{ss}$ of $X$, whose geometric points are the semistable points of $X$, by the action of $G$. 
This GIT quotient $X/\!/G$ contains an open subscheme $X^s/G$ which is a geometric quotient of the scheme $X^s$ of stable points for the linear action.

Associated to the linear action of $G$ on $X$ there is a stratification 
$\{S_\beta : \beta \in \mathcal{B} \}$ of $X$ into disjoint $G$-invariant locally closed subschemes, one of which is $X^{ss}$ \cite{kempf_ness,kirwan}. 
In this paper we consider the problem of finding quotients for each unstable stratum $S_\beta$ separately. For each
$\beta \in \mathcal{B}$ we find a categorical quotient of the $G$-action on the stratum $S_\beta$. 
However this categorical quotient is far from an orbit space in general. We attempt to rectify this by 
making small perturbations to a canonical linearisation on a projective completion $\hat{S}_\beta$ of $S_\beta$ and an associated affine bundle over $S_\beta$ and considering GIT quotients with respect to these perturbed linearisations.

We then apply this to construct moduli spaces of unstable sheaves on a complex projective scheme $W$ which have some additional data (depending on a choice of any sufficiently positive integer $n$) called an $n$-rigidification. There is a well-known construction due to Simpson \cite{simpson} of the moduli space of semistable pure sheaves on $W$ of fixed Hilbert polynomial as the GIT quotient of a linear action of a special linear group $G$ on a scheme $Q$ (closely related to a quot-scheme) which is $G$-equivariantly embedded in a projective space. 
This construction can be chosen so that elements of $Q$ which parametrise sheaves of a fixed Harder--Narasimhan type form a stratum in the stratification of $Q$ associated to the linear action of $G$ (modulo taking connected components of strata).
As above, we consider perturbations of the canonical linearisation on a projective completion of this stratum using a parameter $\theta$ which defines for us a notion of semistability for sheaves of this fixed Harder--Narasimhan type $\tau$. Finally for each $\tau$ we construct a moduli space of S-equivalence classes of $\theta$-semistable $n$-rigidified sheaves of fixed Harder--Narasimhan type $\tau$. 

The layout of this paper is as follows. $\S$\ref{norm sq strat} summarises the properties of the stratifications introduced in \cite{kempf_ness,kirwan} when $X$ is a nonsingular complex projective variety with a linear $G$-action.  In $\S$\ref{how to modify} we construct linearisations on a projective completion of a given stratum in this stratification and provide a categorical quotient of each unstable stratum. In $\S$\ref{ext proj sch} we observe that this construction can be extended without difficulty from varieties to schemes. $\S$\ref{simp constr} summarises Simpson's construction of moduli spaces of semistable sheaves and calculates the associated Hilbert--Mumford functions for one-parameter subgroups, while $\S$\ref{strat of Q} relates the stratification of the parameter scheme $Q$ to Harder--Narasimhan type. In $\S$\ref{nrigidified} we define what we mean by an $n$-rigidified sheaf. Finally in $\S$\ref{moduli unstable sheaves} we construct moduli spaces for $n$-rigidified sheaves of fixed Harder--Narasimhan type which are semistable with respect to a given parameter $\theta$.

\subsection*{Acknowledgements}\label{ackref}
Our thanks go to Dima Arinkin and to the referee for helpful comments on an earlier version of this paper.

\section{Stratifications of $X$}\label{norm sq strat}\markright{stratifications of $X$}

In this section we state the results needed from \cite{kirwan} for linear reductive group actions on nonsingular projective varieties. Let $G$ be a complex reductive group acting linearly on a smooth complex projective variety $X$  with respect to an  ample line bundle $\mathcal{L}$. Abusing notation
we will use ${\mathcal{L}}$ to denote both the linearisation (the lift of the
$G$-action to the line bundle) and the line bundle itself.
For the purposes of GIT we can assume without loss of generality that
$\mathcal{L}$ is very ample, so that $X$ is embedded in a projective space $\mathbb{P}^n = \mathbb{P}(H^0(X, \mathcal{L})^*)$ and the action of  $G$ is given by a homomorphism $\rho : G \rightarrow \mathrm{GL}(n+1)$. 
The associated GIT quotient $X/\!/G = X/\!/_{{\mathcal{L}}} G $ is topologically the semistable set $X^{ss}=X^{ss}(\mathcal{L})$ modulo S-equivalence, where $x$ and $y$ in $X^{ss}$ are S-equivalent if and only if the closures of their $G$-orbits meet in $X^{ss}$.
The fact that $G$ is a complex reductive group means that it is the complexification of a maximal compact subgroup $K$, and we assume without loss of generality that $K$ acts unitarily on $\mathbb{P}^n$ via $\rho : K \rightarrow \mathrm{U}(n+1)$. 

Since $X$ is nonsingular, the Fubini-Study metric on $\mathbb{P}^n$ gives $X$ a K\"ahler structure and the K\"ahler form $\omega$ is a $K$-invariant symplectic form on $X$.  Let $\mathfrak{K}$ denote the Lie algebra of $K$; the action of $K$ on the symplectic manifold $(X,\omega)$ is Hamiltonian with moment map $\mu : X \rightarrow \mathfrak{K}^*$ defined by
\[ \mu(x) := \rho^*\left(\frac{x^{\star}\bar{x}^{\star t}}{2 \pi i ||x^\star||^2}\right) \]
where $x^\star \in \mathbb{C}^{n+1}$ lies over $x \in \mathbb{P}^n$ and $\rho^* : \mathfrak{u}(n+1)^* \rightarrow \mathfrak{K}^*$ is dual to $\mathrm{Lie}\rho$. 
Then $x \in X$ is semistable if and only if the closure of its $G$-orbit meets
$\mu^{-1}(0)$, and the inclusion of $\mu^{-1}(0)$ in $X^{ss}$ induces a homeomorphism from the symplectic quotient $\mu^{-1} (0) / K$ to the GIT quotient $X /\!/ G$.

We fix an inner product on the Lie algebra $\mathfrak{K}$ which is invariant under the adjoint action of $K$, and use it to identify $\mathfrak{K}^*$
with $\mathfrak{K}$. The norm square of the moment map $|| \mu ||^2 : X \rightarrow \mathbb{R}$ with respect to this inner product induces a Morse-type stratification of $X$ into $G$-invariant locally closed nonsingular subvarieties 
\[ X = \bigsqcup_{\beta \in \mathcal{B}} S_{\beta} \]
where the indexing set $\mathcal{B}$ is a finite set of adjoint orbits in the Lie algebra $\mathfrak{K}$ (or equivalently a finite set of points in a fixed
positive Weyl chamber $\mathfrak{t}_+$ in $\mathfrak{K}$). In particular $0 \in \mathcal{B}$ indexes the open stratum $S_0$, which is equal to the semistable subset $X^{ss}$.

\begin{rmk}
It is important to note that this stratification depends on the choice of linearisation and the choice of invariant inner product on $\mathfrak{K}$. However, the stratification is unchanged if the ample line bundle $\mathcal{L}$ is replaced with $\mathcal{L}^{\otimes m}$ for any integer $m>0$, which means that we can work with rational linearisations $\mathcal{L}^{\otimes q}$ for $q \in \mathbb{Q} \cap (0, \infty)$.
\end{rmk}

\begin{rmk} The gradient flow of $|| \mu ||^2$ from any $x \in X$ is contained in the $G$-orbit of $x$, and so the stratification of $X$ is given by intersecting $X$ with the stratification of the  ambient projective space $\mathbb{P}^n$. 
\end{rmk}

\begin{rmk} \label{rmk2.2}
If $X$ is singular (and/or quasi-projective rather than projective) we still get a stratification of $X$ into $G$-invariant locally closed subvarieties, which may be singular, by intersecting $X$ with the stratification of the ambient projective space $\mathbb{P}^n$. Indeed, as we will see in $\S$\ref{ext proj sch}, we can allow $X$ to be any $G$-invariant projective subscheme of $\mathbb{P}^n$ and obtain a stratification of $X$ by intersecting $X$ with the stratification of $\mathbb{P}^n$.
\end{rmk} 

The strata indexed by nonzero $\beta \in \mathcal{B}$ have an inductive description in terms of semistable sets for actions of reductive subgroups of $G$ on subvarieties of $X$ \cite{kirwan}. We fix a maximal torus $T$ of $K$ and let $H := T_{\mathbb{C}}$ be the complexification of $T$, which is a maximal torus of $G = K_{\mathbb{C}}$. We also fix a positive Weyl chamber $\mathfrak{t}_+$ in the Lie algebra $\mathfrak{t}$ of $T$. 
The restriction $\rho|_T : T \rightarrow {\rm U}(n+1)$ is diagonalisable with
weights \[\alpha_0, \dots , \alpha_n : T \rightarrow S^1.\] If we identify the tangent space of $S^1$ at the identity with the line $2\pi i \mathbb{R}$ in the complex plane, and identify $2\pi i \mathbb{R}$ with $\mathbb{R}$ in the natural way, then by taking the derivative of $\alpha_j$ at the identity we get an element of the dual of the Lie algebra $\mathfrak{t}$ which we also call $\alpha_j$. 
The index set $\mathcal{B}$ is defined in \cite{kirwan} to be the set of $\beta \in \mathfrak{t}_+$  such that $\beta$ is the closest point to zero of the convex hull in $\mathfrak{t}$ of some nonempty subset of the set of weights $\{\alpha_0, \dots \alpha_n\}$.

\begin{rmk}
Since the set of weights $\{\alpha_0, \dots \alpha_n\}$ is invariant under the Weyl group, $\mathcal{B}$ can also be identified with the set of $K$-orbits in
$\mathfrak{K}$ of closest points to 0 of convex hulls of subsets of $\{\alpha_0, \dots \alpha_n\}$.
\end{rmk}

If $\beta \in \mathcal{B}$ we define $Z_\beta$ to be
\begin{equation} \label{zedbeta} Z_\beta := X \cap \{[x_o : \dots : x_n] \in \mathbb{P}^n : x_i = 0 \: \mathrm{if} \: \alpha_i \cdot \beta \neq || \beta ||^2 \}.\end{equation}
 $Z_\beta$ also has a symplectic description as the set of critical points for the function $\mu_\beta(x) := \mu(x) \cdot \beta$ on which $\mu_{\beta} $ takes the value $||\beta||^2$. By \cite{kirwan} Lemma 3.15 the critical point set of $||\mu||^2$ is the disjoint union over $\beta \in \mathcal{B}$ of the closed subsets
\begin{equation} \label{crit}  C_\beta := K(Z_\beta \cap \mu^{-1}(\beta)) . \end{equation}
The stratum $S_\beta$ corresponding to the critical point set $C_\beta$ is the set of points in $X$ whose path of steepest descent under $||\mu||^2$ has a limit point in $C_\beta$. 

\begin{rmk}
The stratum $S_\beta$ depends only on the adjoint orbit of $\beta$, but in order to define $Z_\beta$ we need to fix an element in that adjoint orbit.
\end{rmk}

The strata have an alternative algebraic description. Let $\mathrm{Stab} \beta$ be the stabiliser of $\beta$ under the adjoint action of $G$ on its Lie algebra $\mathfrak{g}$; then $Z_\beta$ is $\mathrm{Stab} \beta$-invariant (\cite{kirwan} $\S$4.8). We consider the action of $\mathrm{Stab }\beta$ on $Z_\beta$ with respect to the original linearisation twisted by the character $-\beta$ of 
$\mathrm{Stab} \beta$, so that the semistable set $Z_\beta^{ss}$ with respect to this modified linearisation is equal to the open stratum for the Morse stratification of the function $|| \mu - \beta ||^2$ on $Z_\beta$.   Let 
\begin{equation} \label{ybeta} Y_\beta := X \cap \left\{ [x_0 : \dots : x_n ] \in \mathbb{P}^n : \begin{array}{l} x_i = 0 \: \mathrm{if} \: \alpha_i \cdot \beta < || \beta ||^2  \: \mathrm{and} \: x_i \neq 0 \:  \\ \mathrm{for \: some \: } i \: \mathrm{such \: that \: } \alpha_i \cdot \beta = || \beta ||^2  \end{array}\right\} \end{equation}
be the set of points in $X$ whose corresponding weights are all on the opposite side to the origin of the hyperplane to $\beta$ and such that at least one of the weights lies on the hyperplane to $\beta$. In the symplectic description, $Y_\beta$ is the set of points in $X$ whose path of steepest descent under $\mu_\beta$ has limit in $Z_\beta$. There is an obvious surjection $p_\beta : Y_\beta \rightarrow Z_\beta $ which is a retraction onto $Z_\beta$. We define $Y_\beta^{ss}= p_\beta^{-1}(Z_\beta^{ss})$; then by \cite{kirwan} Theorem 6.18
\[ S_\beta = G Y_\beta^{ss}. \]

The positive Weyl chamber $\mathfrak{t}_+$ corresponds to a choice of positive roots 
 \[ \Phi_+: = \{ \alpha \in \Phi : \alpha \cdot \eta \geq 0 \mathrm{\:for \: all \: } \eta \in \mathfrak{t}_+ \} \]
where $\Phi \subset \mathfrak{t}^*$ is the set of roots coming from the adjoint action of $T$ on $\mathfrak{g}$. This in turn corresponds to a Borel subgroup $B=B_+$ of $G$ such that the Lie algebra $\mathfrak{b}_+$ of $B_+$ is given by
\[ \mathfrak{b}_+:= \mathfrak{h} \oplus \bigoplus_{\alpha \in \Phi_+} \mathfrak{g}_\alpha .\]
For $\beta \in \mathfrak{t}_+$ we construct a parabolic subgroup $P_\beta:= B_+\mathrm{Stab} \beta$ which may also be defined as
\[ P_\beta : = \{ g \in G : \lim_{t \rightarrow - \infty} \exp(it\beta) \: g \: \exp(it\beta)^{-1} \mathrm{\: exists \: in \: } G \}.\]
The subsets $Y_{\beta}^{ss}$ and $Y_\beta$ are $P_\beta$-invariant (see \cite{kirwan} Lemma 6.10) and by \cite{kirwan} Theorem 6.18 there is an isomorphism
\[ S_\beta \cong G \times_{P_\beta} Y_\beta^{ss}.\]

\begin{rmk} \label{KeN}
This stratification can also be described in terms of the work of Kempf and Ness \cite{kempf_ness} and Hesselink \cite{hesselink}. The Hilbert--Mumford criterion gives a test for (semi-)stability in terms of limits of one-parameter subgroups (1-PSs) acting on a given point $x \in X$. Given a 1-PS $\lambda$, we define $\mu(x, \lambda)$ to be the integer equal to the weight of the $\mathbb{C}^*$-action induced by this 1-PS on the fibre $\mathcal{L}_{x_0}$ where $x_0 = \lim_{t \rightarrow 0} \lambda(t) \cdot x$. We call $\mu(x, \lambda)$ the Hilbert--Mumford function and the Hilbert--Mumford criterion states that $x$ is semistable if and only if $\mu(x, \lambda) \geq 0$ for all 1-PSs. A point $x$ is unstable if and only if it fails the Hilbert--Mumford criterion for at least one 1-PS, and there is a notion of an {\it adapted} 1-PS for this point: that is, a non-divisible 1-PS $\lambda$ for which the quantity $\mu(\lambda,x) / || \lambda||$ is minimised.
The set $\wedge^{\mathcal{L}}(x)$ of 1-PSs which are adapted to $x$ is studied by Kempf \cite{kempf}, who shows that $\wedge^{\mathcal{L}}(x)$ is a full conjugacy class of 1-PSs in a parabolic subgroup $P_x$ of $G$. In fact for each $\lambda \in \wedge^{\mathcal{L}}(x)$,
\[ P_x= P(\lambda) : = \{ g \in G : \lim_{t \rightarrow 0} \lambda(t) g \lambda(t)^{-1} \mathrm{\: exists \: in \: } G \}. \]
These sets of 1-PSs give us a stratification of the unstable locus $X -X^{ss}$
\cite{kempf_ness}, which agrees with the stratification $\{S_{\beta}:\beta \in \mathcal{B}\}$ described above, as
follows.

Each $\beta \in \mathcal{B}$ is rational in the sense that there is a natural number $m>0$ such that $m\beta$ defines a 1-PS $\mathbb{C}^* \to H = T_{\mathbb{C}}$ whose restriction to $S^1 \to T$ has derivative at the identity
\[\mathbb{R} \cong 2\pi i \mathbb{R} \cong {\rm Lie}S^1 \to \mathfrak{t}\]
sending 1 to $m\beta$. For any rational $\beta \in \mathfrak{t}$ let $\lambda_\beta: \mathbb{C}^* \to H$ be the unique
non-divisible 1-PS which is defined by $q\beta$ for some positive rational number $q$. Then if $\beta \in \mathcal{B} \setminus \{ 0 \}$ 
we have 
\[P_\beta = P(\lambda_{\beta})\]
and $\lambda_{\beta}$ is a 1-PS adapted to $x$.
\end{rmk}

\section{Quotients of the unstable strata}\label{how to modify}\markright{quotients of the unstable strata}

Let $\beta \in \mathcal{B} \setminus \{ 0 \}$ be a nonzero index 
for the stratification $\{S_\beta: \beta \in \mathcal{B}\}$ and consider the projective completion 
\[ \hat{S_\beta} : =G \times_{P_\beta} \overline{ Y_\beta} \subset G \times_{P_\beta} X \]
of the stratum $S_\beta \cong G \times_{P_\beta} Y_\beta^{ss}$, where $\overline{ Y_\beta}$ is the closure of $Y_\beta^{ss}$ in $X$.

\begin{rmk} \label{disconn} It is always the case that
\[\overline{ Y_\beta} \subseteq X \cap \{ [x_0 : \cdots : x_n] : x_j = 0 \mathrm{\: if \:} \alpha_j \cdot \beta < ||\beta||^2 \}. \]
We often have equality here (for example when $X=\mathbb{P}^n$) but
it might be the case, for example, that  
$X \cap \{ [x_0 : \cdots x_n] : x_j = 0 \mathrm{\: if \:} \alpha_j \cdot \beta < ||\beta||^2 \}$ has connected components which do not meet $Y_\beta$.

It may also be the case that $Z_\beta, Y_\beta$ and $S_\beta$ are disconnected
(cf.  \cite{kirwan} $\S$5),
in which case we can, if we wish, refine the stratification by replacing $Z_\beta$ with its
connected components $Z_{\beta, j}$, say, and setting $S_{\beta,j} = G Y_{\beta,j}^{ss}$ where $Y_{\beta,j}^{ss} = p_\beta^{-1}(Z_{\beta,j}^{ss})$
and $Z_{\beta,j}^{ss} = Z_{\beta,j} \cap Z_{\beta}^{ss}$.
Then each $S_{\beta,j}$ will be a connected component of $S_\beta$ (so long
as $Z_{\beta,j}^{ss}$ is non-empty). In what follows we will work with $S_\beta$
for simplicity of notation,
but we could equally well work with its connected components separately.
\end{rmk}

The action of $P_\beta$ on $X$ extends to an action of $G$ on $X$ so there is a natural isomorphism
\[ G \times_{P_\beta} X \cong G/P_\beta \times X\]
\[ (g,x ) \mapsto (gP_\beta, g \cdot x). \]
In order to find new linearisations on $\hat{S_\beta}$ we can consider linearisations on $G \times_{P_\beta} X \cong G/P_\beta \times X$ and restrict them to $\hat{S_\beta}$. The quotient $G/P_\beta$ is a partial flag variety and linearisations of the $G$-actions on such varieties are well understood.

\subsection{Line bundles on partial flag varieties $G /P$}\label{gen parabolic}

We review the construction of line bundles on partial flag varieties; for more detailed information see \cite{flag,l_s_3,l_s_4,L_s_5}. 

For the moment we assume that $G$ is semisimple and simply connected. Fix sets of positive roots $\Phi_+ \subset \Phi$ and simple roots $\Pi$. Let $\omega_i$ denote the fundamental dominant weight associated to a simple root $\alpha_i$. If $\lambda = \sum a_i \omega_i$ is a dominant weight then define
\[ \Pi_\lambda : = \{ \alpha_i \in \Pi : a_i = 0 \} \subset \Pi .\]
Let $\lambda$ also denote the corresponding one-parameter subgroup; then the parabolic subgroup $P(\lambda)$ associated to the 1-PS $\lambda$ has associated simple roots 
\[ \Pi_{P(\lambda)} := \{ \alpha_i \in \Pi: -\alpha_i \mathrm{\: is \: a \: root \: of \:} P(\lambda) \} \subset \Pi \]
and these sets agree, so that $\Pi_\lambda = \Pi_{P(\lambda)}$.

A character $\chi : H \rightarrow \mathbb{C}^* $ extends to $P(\lambda)$ if and only if $\chi \cdot \alpha^\vee = 0$ for all coroots $\alpha^\vee$ such that $\alpha \in \Pi_{P(\lambda)}$. The weights naturally correspond to characters and the character defined by $\lambda$ extends to $P(\lambda)$ since
\[ \lambda \cdot \alpha_i^\vee = a_i = 0 \mathrm{\: \: for \: \: all \: \:} \alpha_i \in \Pi_{P(\lambda)} \]
by the definition of this set.
We let $\lambda$ also denote the associated character of $P(\lambda)$ and define a line bundle $\mathcal{L}(\lambda)$ on $G/P(\lambda)$ to be the line bundle associated to the character $\lambda^{-1}$; that is,
\[ \begin{array}{c} \mathcal{L}(\lambda) := G \times_{P(\lambda)} \mathbb{C} \\ \downarrow \\ G/P(\lambda) \end{array} \]
where $(g,z)$ and $ (gp, \lambda(p)z)$ are identified for all $p \in P(\lambda)$. The sections of $ \mathcal{L}(\lambda)$ are given by
\[ H^0( G / P(\lambda),\mathcal{L}(\lambda)) = \{ f: G \rightarrow \mathbb{C} : f(gp) =\lambda(p)f(g) \mathrm{\: for \: all \;} g \in G, p \in P \} \]
and the natural left $G$-action gives this vector space a $G$-module structure. Let $V(\lambda)$ denote the representation of $G$ of highest weight $\lambda$. By the Borel--Weil--Bott theorem \cite{bott}, there is an isomorphism of $G$-modules
\[ H^0( G / P(\lambda),\mathcal{L}(\lambda)) \cong V(\lambda)^*. \]

The line bundle $\mathcal{L}(\lambda)$ is very ample if and only if 
\[ \lambda \cdot \alpha_i^\vee = a_i > 0 \mathrm{\:for \: all \:} \alpha_i \notin \Pi_{P(\lambda)} \]
which is clearly the case by definition of $\Pi_\lambda = \Pi_{P(\lambda)}$. Thus there is an embedding
\[ G/P(\lambda) \hookrightarrow \mathbb{P}( H^0( G / P(\lambda),\mathcal{L}(\lambda))^*) \cong \mathbb{P}(V(\lambda))\]
which is the natural projective embedding of the partial flag variety $G/P(\lambda)$. More concretely, let $v_{\max}$ denote the highest weight vector in $V(\lambda)$, so that $v_{\max}$ is an eigenvector for the action of $T$ with eigenvalue $\lambda$; then the embedding is given by the inclusion of the orbit $G\cdot v_{\max}$,
\[ G/P(\lambda) \hookrightarrow \mathbb{P}(V(\lambda))  \]
\[ gP(\lambda) \mapsto [g \cdot v_{\max}].\] 

\begin{rmk}
We will be primarily interested in the case when $G$ is a subgroup of $\mathrm{GL}(n)$ and the weight $\lambda$ is restricted from
$\mathrm{GL}(n)$, and here we do not need to assume that $G$ is
simply connected or semisimple. For we can view the weight $\lambda$ as an
element  of the dual of the Lie algebra of both $\mathrm{GL}(n)$ and 
$\mathrm{PGL}(n)$ or equivalently $\mathrm{SL}(n)$. There are associated parabolics $P(\lambda_{\mathrm{GL}})$ and $P(\lambda_{\mathrm{SL}})$,
and the partial flag varieties for these two parabolics agree
\[ \mathrm{GL}(n) / P(\lambda_{\mathrm{GL}})  = \mathrm{SL}(n) / P(\lambda_{\mathrm{SL}}). \]
Since $\mathrm{SL}(n)$ is semisimple and simply connected there is a projective embedding of this partial flag variety into $\mathbb{P}(V(\lambda_{\mathrm{SL}}))$ where $V(\lambda_{\mathrm{SL}})$ is the representation of $\mathrm{SL}(n)$ with highest weight $\lambda_{\mathrm{SL}}$. 
We have $G \subset \mathrm{GL}(n)$ and $P(\lambda) = G \cap P(\lambda_{\mathrm{GL}})$ and so
\[ G/P(\lambda) \subseteq \mathrm{GL}(n) / P(\lambda_{\mathrm{GL}})  = \mathrm{SL}(n) / P(\lambda_{\mathrm{SL}}). \]
Hence we can use this inclusion and the embedding of $\mathrm{SL}(n) / P(\lambda_{\mathrm{SL}})$ described above to obtain a projective embedding of $G/P(\lambda)$.
\end{rmk}

\subsection{The canonical linearisation on $\hat{S_\beta}$}\label{nat linear on hat var}

We have seen that given a one-parameter subgroup $\lambda$ of $G$ as above
there is a natural ample linearisation of the $G$-action on the partial flag variety $G/P(\lambda)$. We can apply this to the case when the parabolic subgroup is $P_\beta = P(\lambda_{\beta})$. The natural embedding of the partial flag variety $G/P_\beta$ is thus given by the very ample line bundle $\mathcal{L}(\lambda_{\beta})$
\[ G/P_\beta \hookrightarrow \mathbb{P}( H^0( G / P_\beta,\mathcal{L}(\lambda_{\beta}))^*) \cong \mathbb{P}(V(\beta)).\]

Let $\mathcal{L}_\beta$ denote the $G$-linearisation on $G/P_\beta \times X$ given by the tensor product of the pullbacks of $\mathcal{L}(\lambda_{-\beta})$ on $G/P_\beta$ and $\mathcal{L}$ on $X$ to $G/P_\beta \times X$. We also let $\mathcal{L}_\beta$ denote the restriction of this linearisation to $\hat{S}_\beta$ and call this the canonical linearisation. There is also a canonical linearisation $\mathcal{L}_\beta$ of the $\mathrm{Stab} \beta$-action on $Z_\beta$ given by twisting the original linearisation $\mathcal{L}$ by the character of $\mathrm{Stab} \beta$ corresponding to $-\beta$. Recall that $Z_\beta^{ss}$ is defined to be the semistable subset for this linearisation. The character of $\mathrm{Stab} \beta$ corresponding to $-\beta$ extends to a character of $P_\beta$ and so there is also a canonical linearisation $\mathcal{L}_\beta$ of the $P_\beta$-action (or the $\mathrm{Stab} \beta$-action) on $Y_\beta$ given by twisting $\mathcal{L}$ by the character corresponding to $-\beta$. All of these linearisations are equal to the restriction of the canonical $G$-linearisation $\mathcal{L}_\beta$ on $\hat{S}_\beta$ to the relevant subvarieties and subgroups. The following lemma explains why we call $\mathcal{L}_\beta$ the canonical linearisation.

\begin{lemma}\label{lots of isos}
We have isomorphisms of graded algebras
\[ \bigoplus_{r \geq 0} H^0(\hat{S}_\beta, \mathcal{L}_\beta^{\otimes r})^G \cong \bigoplus_{r \geq 0} H^0(\overline{Y}_\beta, \mathcal{L}_\beta^{\otimes r})^{P_\beta} \cong \bigoplus_{r \geq 0} H^0({Z}_\beta, \mathcal{L}_\beta^{\otimes r})^{\mathrm{Stab} \beta}.\]
\end{lemma}
\begin{proof}
The first isomorphism follows from the fact that $\hat{S}_\beta = G \times_{P_\beta} \overline{Y}_\beta$ and the canonical $G$-linearisation on $\hat{S}_\beta$ is equal to $G \times_{P_\beta} \mathcal{L}_\beta$ where here $\mathcal{L}_\beta$ is the canonical $P_\beta$-linearisation on $\overline{Y}_\beta$.

Let $\lambda_\beta : \mathbb{C}^* \rightarrow G$ be the 1-PS determined by the rational weight $\beta$. Then $\lambda_\beta(\mathbb{C}^*) \subseteq P_\beta$ and so
\[ \bigoplus_{r \geq 0} H^0(\overline{Y}_\beta, \mathcal{L}_\beta^{\otimes r})^{P_\beta} \subseteq \bigoplus_{r \geq 0} H^0(\overline{Y}_\beta, \mathcal{L}_\beta^{\otimes r})^{\lambda_\beta(\mathbb{C}^*)}. \]
The torus $\lambda_\beta(\mathbb{C}^*)$ acts on $\overline{Y}_\beta$ with respect to the canonical linearisation $\mathcal{L}_\beta$ with non-negative weights, and has zero weights exactly on $Z_\beta$. Hence
\[ \bigoplus_{r \geq 0} H^0(\overline{Y}_\beta, \mathcal{L}_\beta^{\otimes r})^{\lambda_\beta(\mathbb{C}^*)} \cong \bigoplus_{r \geq 0} H^0(Z_\beta, \mathcal{L}_\beta^{\otimes r})^{\lambda_\beta(\mathbb{C}^*)} \]
and so
\[ \bigoplus_{r \geq 0} H^0(\overline{Y}_\beta, \mathcal{L}_\beta^{\otimes r})^{P_\beta} \subseteq \bigoplus_{r \geq 0} H^0(\overline{Y}_\beta, \mathcal{L}_\beta^{\otimes r})^{\mathrm{Stab} \beta} \cong \bigoplus_{r \geq 0} H^0({Z}_\beta, \mathcal{L}_\beta^{\otimes r})^{\mathrm{Stab} \beta}.\]
Let $\sigma \in H^0({Z}_\beta, \mathcal{L}_\beta^{\otimes r})^{\mathrm{Stab} \beta}$ and consider $p_\beta^* \sigma \in H^0({Y}_\beta, \mathcal{L}_\beta^{\otimes r})^{\mathrm{Stab} \beta}$ where $p_\beta : Y_\beta \rightarrow Z_\beta$ is the retraction defined by $\beta$. We have that $P_\beta = \mathrm{Stab} \beta U_\beta$ where $U_\beta$ is the unipotent radical of $P_\beta$ and there is a retraction $q_\beta : P_\beta \rightarrow \mathrm{Stab} \beta $ such that
\begin{equation} \label{eqqp} p_{\beta} (p \cdot y) = q_{\beta} (p) \cdot p_{\beta} (y) \end{equation}
for all $y \in Y_\beta$ and $p \in P_\beta$.
The action of $P_\beta$ on $H^0({Y}_\beta, \mathcal{L}_\beta^{\otimes r})$ 
is induced from its action on $Y_\beta$ and $\mathcal{L}_\beta$,
and so if $p \in P_\beta$ we have
\[ p \cdot p_\beta^* \sigma  = p_\beta^* (q_\beta(p) \cdot \sigma) = p_\beta^* \sigma \] 
as $\sigma$ is $\mathrm{Stab} \beta$ invariant. Therefore,
\[ \bigoplus_{r \geq 0} H^0(\overline{Y}_\beta, \mathcal{L}_\beta^{\otimes r})^{P_\beta} \cong \bigoplus_{r \geq 0} H^0({Z}_\beta, \mathcal{L}_\beta^{\otimes r})^{\mathrm{Stab} \beta} .\]
\end{proof}

\begin{rmk}
Unfortunately if $\beta \neq 0$ then $\mathcal{L}(\lambda_{-\beta})$ is a non-ample linearisation of the $G$-action on $G/P_\beta$, and the canonical $G$-linearisation $\mathcal{L}_\beta$ on $\hat{S}_\beta$ is in general non-ample too, as the following example shows.
\end{rmk}

\begin{ex} \label{exx3.5}
Consider $G = \mathrm{SL}(2, \mathbb{C})$ acting on the complex projective line $X = \mathbb{P}^1$ with respect to $\mathcal{L} = \mathcal{O}_{\mathbb{P}^1}(1)$. The semistable set is empty and the action is transitive so there will be one nonzero index in the stratification of $X$. We choose a maximal torus $T = \{ \mathrm{diag}(t, t^{-1}) : t \in S^1 \}$; then the weights of $T$ acting on $\mathbb{C}^2$ are $\alpha_0 = \alpha, \alpha_1 = \alpha^{-1}$ where
\[ \alpha : T \rightarrow S^1 \]
\[ \left( \begin{array}{cc} t & 0 \\0 & t^{-1} \end{array} \right) \mapsto t. \]
The Lie algebra of $T$ is $\mathfrak{t} \cong \mathbb{R}$ and we pick the positive Weyl chamber  $\mathfrak{t}_+$ which contains $\alpha$. Then $\beta = \alpha$ is an index for the stratification of $X$ and we have that
\[ Z_\beta = Z_\beta^{ss} = Y_\beta = Y_\beta^{ss} = \{ [1:0] \} \]
and $S_\beta = X$. The parabolic subgroup $P_\beta$ is the Borel subgroup of upper triangular matrices and we have an isomorphism \[G/P_\beta \cong \mathbb{P}^1\]
\[ gP_\beta \mapsto  g \cdot \left(\begin{array}{c} 1 \\ 0 \end{array} \right). \]
The very ample line bundle on $G/P_\beta$ defined by $\beta$ is $\mathcal{O}_{\mathbb{P}^1}(1)$ and the line bundle defined by $-\beta$ is $\mathcal{O}_{\mathbb{P}^1}(-1)$. The canonical linearisation is given by restricting $\mathcal{O}_{\mathbb{P}^1}(-1) \otimes \mathcal{O}_{\mathbb{P}^1}(1) $ on $\mathbb{P}^1 \times \mathbb{P}^1 \cong G/P_\beta \times X$ to $ \hat{S}_\beta \cong S_\beta = X$. The morphism $S_\beta \rightarrow \mathbb{P}^1 \times \mathbb{P}^1$ is the diagonal morphism and so the canonical linearisation on $S_\beta=X$ is $\mathcal{L}_\beta =\mathcal{O}_{\mathbb{P}^1}$. 
\end{ex}

\begin{prop} \label{mmore}
The projective variety $Z_\beta /\!/_{\mathcal{L}_\beta} \mathrm{Stab} \beta$ is a categorical quotient for the action of
\begin{enumerate}
\renewcommand{\labelenumi}{\roman{enumi})}
\item $\mathrm{Stab} \beta $ on $Z_\beta^{ss}$,
\item $\mathrm{Stab} \beta $ on $Y_\beta^{ss}$,
\item $P_\beta $ on $Y_\beta^{ss}$,
\item $G$ on $S_\beta$.
\end{enumerate}
\end{prop}
\begin{proof} The natural morphism  $Z_\beta^{ss} \rightarrow Z_\beta /\!/_{\mathcal{L}_\beta} \mathrm{Stab} \beta$ is a  
categorical quotient by classical GIT since $\mathcal{L}_\beta$ is ample on 
$Z_\beta$ and $\mathrm{Stab} \beta$ is reductive, so (i) is proved. 

There is a surjective morphism $Y_\beta^{ss} \rightarrow Z_\beta /\!/_{\mathcal{L}_\beta} \mathrm{Stab} \beta$ given by the composition of the retraction $p_\beta : Y_\beta^{ss} \rightarrow Z_\beta^{ss}$ with the categorical quotient $Z_\beta^{ss} \rightarrow Z_\beta /\!/_{\mathcal{L}_\beta} \mathrm{Stab} \beta$. Moreover this surjective morphism $Y_\beta^{ss} \rightarrow Z_\beta /\!/_{\mathcal{L}_\beta} \mathrm{Stab} \beta$ is $P_\beta$-invariant 
by (\ref{eqqp}) and the $\mathrm{Stab} \beta$-invariance of $Z_\beta^{ss} \rightarrow Z_\beta /\!/_{\mathcal{L}_\beta} \mathrm{Stab} \beta$. Thus to prove (ii) and (iii) it suffices to show that any
$\mathrm{Stab} \beta$-invariant morphism $f : Y_\beta^{ss} \rightarrow Y$ factors through $Z_\beta /\!/ _{\mathcal{L}_\beta} \mathrm{Stab} \beta$. As $f$ is $\mathrm{Stab} \beta$-invariant it is constant on orbit closures and so $f = f|_{Z_\beta^{ss}} \circ p_\beta$. Since 
% and
 $f|_{Z_\beta^{ss}} : Z_\beta^{ss} \rightarrow Y$ is $\mathrm{Stab} \beta$-invariant, there is a morphism $h : Z_\beta /\!/ _{\mathcal{L}_\beta} \mathrm{Stab} \beta \rightarrow Y$ such that $f|_{Z_\beta^{ss}}$ is the composition of $h$ with the categorical quotient $Z_\beta^{ss} \rightarrow  Z_\beta /\!/_{\mathcal{L}_\beta} \mathrm{Stab} \beta$ of the $\mathrm{Stab} \beta$-action on $Z_\beta^{ss}$. Then we have a commutative diagram
\begin{center}{$
\begin{diagram} \node{Y_\beta^{ss}} \arrow{e,t}{p_\beta} \arrow{se,b}{f} \node{Z_\beta^{ss}} \arrow{e,t}{} \arrow{s,b}{f|}  \node{Z_\beta /\!/ _{\mathcal{L}_\beta} \mathrm{Stab} \beta} \arrow{sw,b}{h} \\ \node{} \node{Y} \end{diagram}
$}\end{center}
where $f| = f|_{Z_\beta^{ss}}$
and the morphism $f$ factors through $Z_\beta /\!/ _{\mathcal{L}_\beta} \mathrm{Stab} \beta$ as required. 

Thus (ii) and (iii) are proved, and 
(iv) now follows immediately from the fact that $S_\beta \cong G \times_{P_\beta} Y_\beta^{ss}$.
\end{proof}

\begin{rmk} \label{rmkk3.7}
From Lemma \ref{lots of isos} and Proposition \ref{mmore} we see that 
$Z_\beta /\!/_{\mathcal{L}_\beta} \mathrm{Stab} \beta$ has properties we would
like and expect for a GIT quotient of the actions of $\mathrm{Stab} \beta$ and $P_\beta$ on $\overline{Y}_\beta$ and of $G$ on $\hat{S}_\beta$ with respect to the linearisation $\mathcal{L}_\beta$. The linearisation $\mathcal{L}_\beta$ is
ample on $\overline{Y}_\beta$ and the proofs above do indeed show that $Y_\beta^{ss}$ is the semistable set for this linear action of $\mathrm{Stab} \beta$ and that the GIT quotient is $Z_\beta /\!/_{\mathcal{L}_\beta} \mathrm{Stab} \beta$. However the parabolic subgroup $P_\beta$ of $G$ is not
usually reductive and the linearisation $\mathcal{L}_\beta$ is not in general
ample on $\hat{S}_\beta$, so we cannot apply classical GIT to the actions of 
$P_\beta$ on $\overline{Y_\beta}$ and $G$ on $\hat{S}_\beta$ with respect to the linearisation $\mathcal{L}$. For a linear action of a reductive group $G$ on a variety $X$ with respect to a non-ample line bundle, Mumford does define in
\cite{mumford} a notion of semistability and shows that the resulting semistable set
$X^{ss}$ has a categorical quotient; however according to his definition 
for the linearisation $\mathcal{L}_\beta$ on $\hat{S}_\beta$ we would not in general get $\hat{S}_\beta^{ss} = S_\beta$ with the categorical quotient being
$Z_\beta /\!/_{\mathcal{L}_\beta} \mathrm{Stab} \beta$. Indeed in Example \ref{exx3.5} Mumford's semistable set and categorical quotient are empty.
\end{rmk}

\begin{rmk}\label{rmk on mod lin} 
The categorical quotient $S_\beta \rightarrow Z_\beta /\!/_{\mathcal{L}_\beta} \mathrm{Stab} \beta$ collapses more orbits than we might like, resulting in the GIT quotient having lower dimension than expected. This happens because if $y \in Y_\beta^{ss}$ then $p_\beta(y) \in \overline{\mathrm{Stab} \beta \cdot y} \subseteq \overline{G \cdot y}$, and so in the quotient every point in $Y_\beta^{ss}$ is identified with its projection to $Z_\beta^{ss}$.
\end{rmk}

\subsection{Perturbations of the canonical linearisation}\label{how to perturb}

To resolve the issue mentioned in Remark \ref{rmk on mod lin} above we would like to perturb the canonical linearisation $\mathcal{L}_\beta$ for the action 
of $G$ on $\hat{S}_\beta$ or the action of $P_\beta$ on $\overline{Y_\beta}$
and take a GIT quotient with respect to this perturbed linearisation. Unfortunately, as we observed in Remark
\ref{rmkk3.7}, on $\hat{S}_\beta$ the canonical $G$-linearisation is not ample, whereas on $\overline{Y}_\beta$ it is ample, but $P_\beta$ is not reductive, and so in each case applying GIT is delicate. On the other hand $\mathrm{Stab} \beta$ is reductive and $\mathcal{L}_\beta$ is an ample $\mathrm{Stab} \beta$-linearisation on $\overline{Y}_\beta$, so we can try perturbing this linearisation.

\begin{rmk} 
Note that although $\overline{Y}_\beta /\!/_{\mathcal{L}_\beta} \mathrm{Stab} \beta \cong Z_\beta /\!/_{\mathcal{L}_\beta} \mathrm{Stab} \beta$ is a categorical quotient for the $G$-action on $S_\beta$
by Proposition \ref{mmore}, after a perturbation we would no longer expect the GIT quotient $\overline{Y}_\beta /\!/ \mathrm{Stab} \beta $ to give us a categorical quotient of the $G$-action on an open subset of $S_\beta$. Instead,
 if $U$ is a $\mathrm{Stab} \beta$-invariant open subscheme of $Y_\beta^{ss}$, then a categorical quotient for the $\mathrm{Stab} \beta$-action on $U$ will be a categorical quotient for the $G$-action on $G \times_{\mathrm{Stab} \beta}  U$. 
Moreover, since $S_\beta \cong G \times_{P_\beta} Y_\beta^{ss}$, we have a surjective morphism
\[ G \times_{\mathrm{Stab} \beta} Y_\beta^{ss} \rightarrow S_\beta \]
\[ [g, y] \mapsto g \cdot y\]
with fibres isomorphic to $P_\beta / \mathrm{Stab} \beta \cong U_\beta$, the unipotent radical of $P_\beta$, which as an algebraic variety is isomorphic to an affine space. 
\end{rmk}

Recall that the canonical $\mathrm{Stab} \beta$-linearisation $\mathcal{L}_\beta$ on $\overline{Y}_\beta$ is ample and is equal to $\mathcal{L}$ twisted by the character of $\mathrm{Stab} \beta$ associated to $-\beta$. Therefore, to perturb this linearisation we can perturb the original linearisation $\mathcal{L}$ and/or make a perturbation of the character by using $-(\beta + \epsilon \beta')$ rather than $-\beta$ where $\beta' \in \mathfrak{t}_+$ is a rational weight and $\epsilon$ is a small rational number. 

The norm square of the moment map associated to the canonical $\mathrm{Stab} \beta$-linearisation $\mathcal{L}_\beta$ on $\overline{Y}_\beta$ gives us a stratification
\[ \overline{Y}_\beta = \bigsqcup_{\delta \in \hat{\mathcal{B}}_\beta} S_\delta^{\mathrm{can}} \]
of $\overline{Y}_\beta$ such that $S_0^{\mathrm{can}} = Y_\beta^{ss}$. A perturbation of this linearisation also has an associated moment map which gives us a new stratification
\[ \overline{Y}_\beta = \bigsqcup_{\gamma \in \hat{\mathcal{B}}_\beta^{\mathrm{per}}} S_\gamma^{\mathrm{per}} \]
such that $S_0^{\mathrm{per}} \subseteq S_0^{\mathrm{can}} = Y_\beta^{ss}$. The next proposition shows that provided the perturbation is sufficiently small, the second stratification is a refinement of the first stratification. In particular this proposition shows that there is a subset \[{\mathcal{B}}_\beta^{\mathrm{per}} \subset \hat{\mathcal{B}}_\beta^{\mathrm{per}}\] such that
\[ {Y_\beta^{ss}} = \bigsqcup_{\gamma \in {\mathcal{B}}_\beta^{\mathrm{per}}} S_\gamma^{\mathrm{per}}. \]

\begin{prop}\label{small per refines strat}
Let $X$ be a projective variety with a $G$-action and ample linearisation $\mathcal{L}$ and let $\mathcal{L}^{\mathrm{per}}$ be an ample perturbation of this linearisation. If $\mu$ (respectively $\mu_{\mathrm{per}}$) denotes the moment map associated to $\mathcal{L}$ (respectively $\mathcal{L}^{\mathrm{per}}$), then provided $\mathcal{L}^{\mathrm{per}}$ is a sufficiently small perturbation of $\mathcal{L}$ the stratification
\[ X = \bigsqcup_{\gamma \in {\mathcal{B}}^{\mathrm{per}}} S_\gamma^{\mathrm{per}} \]
associated to $|| \mu_{\mathrm{per}} ||^2$ is a refinement of the stratification
\[ X = \bigsqcup_{\beta \in \mathcal{B}} S_\beta \]
associated to $|| \mu ||^2$.
\end{prop}
\begin{proof}
Fix a maximal torus $H = T_\mathbb{C} \subseteq G$ and consider its fixed point set $X^H$ which has a finite number of connected components $F_i$ for $i \in I$. Let $\alpha_i$ (respectively $\alpha_i^{\mathrm{per}}$) denote the weight with which $T$ acts on $\mathcal{L}|_{F_i}$ (respectively on $\mathcal{L}^{\mathrm{per}}|_{F_i}$). Then by definition $\mathcal{B}$ is the set of closest points to $0$ of convex hulls of  subsets of $\{ \alpha_i : i \in I \}$ modulo the action of the Weyl group $W$. Similarly ${\mathcal{B}}^{\mathrm{per}}$ is the set of closest points to $0$ of  convex hulls of  subsets of $\{ \alpha^{\mathrm{per}}_i : i \in I \}$ modulo the $W$-action. Fix $\gamma\in \mathfrak{t}$ representing a point of 
${\mathcal{B}}_\beta^{\mathrm{per}}$, so that $\gamma$ is the  closest point to $0$ of the convex hull of 
\[\{ \alpha_i^{\mathrm{per}} : i \in I \: \mathrm{and} \: \alpha_i^{\mathrm{per}} \cdot \gamma \geq || \gamma ||^2 \}\] and we can list these weights as $\alpha_{i_0}^{\mathrm{per}}, \dots ,  \alpha_{i_k}^{\mathrm{per}}$, say. We define $\beta_\gamma \in \mathcal{B}$ to be the $W$-orbit of the closest point to zero of the convex hull of 
\[ \{ \alpha_{i_0}, \dots ,  \alpha_{i_k}\}.\]
 As the linearisation $\mathcal{L}^{\mathrm{per}}$ becomes close to $\mathcal{L}$ the weight $\alpha_i^{\mathrm{per}}$ becomes close to
$\alpha_i$ for each $i$ and so $\gamma$ approaches $\beta_\gamma$. 
We need to show that if this perturbation is sufficiently small then
\[ S_\beta = \bigsqcup_{\begin{array}{c} \gamma \in {\mathcal{B}}^{\mathrm{per}} \\ \beta = \beta_\gamma \end{array}} S_\gamma^{\mathrm{per}}. \]
Since $\{S_\beta: \beta \in \mathcal{B}\}$ and
$\{ S_\gamma^{\mathrm{per}}: \gamma \in {\mathcal{B}}^{\mathrm{per}}\}$ are both stratifications of $X$, it suffices to show that 
\[S_\gamma^{\mathrm{per}} \subseteq S_{\beta_\gamma}\]
 for all $\gamma \in {\mathcal{B}}^{\mathrm{per}}$, and for this it is
enough to show that 

(i) $S_\gamma^{\mathrm{per}} \cap S_{\beta'} = \phi$ for all $\beta' \in \mathcal{B}$ such that $|| \beta' || > || \beta_\gamma||$, and

(ii) $Y_\gamma^{\mathrm{per}} \subset  \overline{Y_{\beta_\gamma}}$, 

\noindent since then 
$S_\gamma^{\mathrm{per}}=GY_\gamma^{ss, \: \mathrm{per}} \subset G \overline{Y_{\beta_\gamma}} \setminus \cup_{||\beta'|| > ||\beta_\gamma||} S_{\beta'} = S_{\beta_\gamma}$ as required.

Firstly we consider how small the perturbation must be for (i) and (ii) to hold. Let
\[ \epsilon_0 := \min \left\{  || \beta ||^2 - \alpha_i \cdot \beta : \beta \in \mathcal{B}, i \in I \: \mathrm{such \: that} \:  || \beta ||^2 > \alpha_i \cdot \beta  \right\} \]
and
\[ \epsilon_1 := \min \left\{ \mid || \beta' || - || \beta || \mid : \beta', \beta \in \mathcal{B} \: \mathrm{and} \:  || \beta' || \neq || \beta || \right\}. \]
Then $\epsilon_0>0$ and $\epsilon_1 >0$ depend only on the initial linearisation $\mathcal{L}$ of the $G$-action on $X$. Since $X$ is compact $M = \sup\{ || \mu(x) || : x \in X \}$ exists and we can define
\[ \epsilon := \min \left\{ 1, \frac{\epsilon_0}{4M+1} , \frac{\epsilon_1}{3} \right\} > 0. \]
If the perturbation $\mathcal{L}^{\mathrm{per}}$ is sufficiently small then
\begin{itemize}
\item[(a)] for all $\gamma \in \mathcal{B}^{\mathrm{per}}$ we have $|| \gamma - \beta_{\gamma} || < \epsilon$, and
\item[(b)] for all $ x \in X$ we have $|| \mu(x) - \mu_{\mathrm{per}}(x) || < \epsilon$;
\end{itemize}
we will assume that these conditions are satisfied.

\emph{Proof of} (i): Suppose that $\gamma \in {\mathcal{B}}^{\mathrm{per}}$ and
$\beta' \in \mathcal{B}$ and $|| \beta' || > || \beta_\gamma||$. If $y \in S_\gamma^{\mathrm{per}}$ then by (\ref{crit}) there exists $g \in G$ such that $gy$ is arbitrarily close to some point $x$ in $Z^\mathrm{per}_\gamma \cap \mu^{-1}_{\mathrm{per}}(\gamma)$, so there exists 
$g \in G$ such that
\[  || \mu_{\mathrm{per}}(gy) || - || \gamma || < \epsilon .\]
 Then (b) implies that $ || \mu(gy) || - || \mu_{\mathrm{per}}(gy) ||  < \epsilon $ and (a) implies that $|| \gamma|| - || \beta_\gamma || < \epsilon$ so that $|| \mu(gy) || < 3\epsilon + || \beta_\gamma||$. However by the
definition of $\epsilon$ we know that $3 \epsilon \leq || \beta' || - || \beta_{\gamma}||$, so we conclude that $|| \mu(gy) ||  < || \beta'||$ which implies $gy \notin S_{\beta'}$, and so $y$ does not belong to $S_{\beta'}$.

\emph{Proof of} (ii): Let $y \in Y_\gamma^{\mathrm{per}}$ where
$\gamma \in {\mathcal{B}}^{\mathrm{per}}$,
and consider its gradient flow under the 1-PS associated to $\beta_\gamma$,
which has limit point $x$, say. Then $x \in \overline{ Y_\gamma^{\mathrm{per}}}$ since
$x$ is in the $H$-orbit closure of $y$ and $Y_\gamma^{\mathrm{per}}$ is invariant under $H$, and hence 
\begin{equation} \label{star} \mu_{\mathrm{per}}(x) \cdot \gamma \geq || \gamma||^2.\end{equation} 
Note that
\[ \mu_{\mathrm{per}}(x) \cdot \gamma - \mu(x) \cdot \beta_\gamma = ( \mu_{\mathrm{per}}(x) - \mu(x)) \cdot \gamma + \mu(x) \cdot (\gamma - \beta_\gamma); \]
the assumption (b) implies that $| ( \mu_{\mathrm{per}}(x) - \mu(x)) \cdot \gamma | < \epsilon || \gamma||$ and (a) together with the inequality $|| \mu(x) || \leq M$ implies that $|  \mu(x) \cdot (\gamma - \beta_\gamma)| < M \epsilon$, so that
\begin{equation} \label{dagger} |\mu_{\mathrm{per}}(x) \cdot \gamma - \mu(x) \cdot \beta_\gamma |< \epsilon (M + || \gamma||). \end{equation}
To prove that $y \in \overline{Y_{\beta_\gamma}}$ (at least interpreted as in
Remark \ref{disconn}, which is sufficient for the purposes of this proof)
it suffices to show that $\mu(x).\beta_\gamma > \alpha_i.\beta_\gamma$ for
all $i$ such that $\alpha_i.\beta_\gamma < ||\beta_\gamma||^2$, so it is
enough to show that
\[\mu(x).\beta_\gamma > ||\beta_\gamma||^2 - \epsilon_0.\]
Combining (\ref{star}) and (\ref{dagger}) gives 
$\mu(x) \cdot \beta_\gamma > || \gamma||^2 -  \epsilon (M + || \gamma||)$ and 
so by (a) we get the following inequality
\[\mu(x) \cdot \beta_\gamma > || \beta_\gamma||^2 -  \epsilon (M + || \gamma||+ 2|| \beta_\gamma||) .\]
Again using (a) we have that $ - \epsilon || \gamma|| > -\epsilon|| \beta_\gamma|| - \epsilon^2$ and since $||\beta_\gamma|| \leq M$ we see that
\[ \mu(x) \cdot \beta_\gamma > || \beta_\gamma||^2 - (4M + \epsilon)\epsilon  .\] 
By the choice of $\epsilon$ we know that $(4M+\epsilon)\epsilon \leq 
(4M+1)\epsilon \leq \epsilon_0$ and so
\[  \mu(x) \cdot \beta_\gamma > || \beta_\gamma||^2 -  \epsilon_0 \]
as required.
This completes the proof of (ii) and hence of the proposition.
\end{proof}

\section{Extending to projective schemes} \label{ext proj sch}\markright{extending to projective schemes}

In this section we observe that  the constructions in the previous sections
for nonsingular projective varieties can be extended to the case when 
$X$ is any projective scheme with an ample $G$-linearisation $\mathcal{L}$. For this it is enough to deal with the case when $\mathcal{L}$ is very ample and check that the resulting constructions do not change when $\mathcal{L}$ is replaced
with $\mathcal{L}^{\otimes m}$ for any positive integer $m$. 

Thus let us assume
that
$X$ is a closed subscheme of $\mathbb{P}^n$ and the action of $G$ on $X$ is given by a linear representation $G \rightarrow \mathrm{GL}(n+1)$. For the $G$-action on the ambient projective space $\mathbb{P}^n$ we can define the subvarieties $Z_\beta^{ss}$ and $Y_\beta^{ss}$ as before. We can also define the closed subvariety $\overline{Y}_\beta$ of $\mathbb{P}^n$ and use the scheme structure on $\mathbb{P}^n$ to give this the reduced induced closed scheme structure as in \cite{hartshorne}, II Example 3.2.6. This gives $\hat{S}_\beta := G \times_{P_\beta} \overline{Y}_\beta$ its scheme structure. Then the open subsets $S_\beta \subset \hat{S}_\beta$ and $Y_\beta^{ss} \subset \overline{Y}_\beta$ get an induced scheme structure as open subsets of schemes. We have a stratification
\[ \mathbb{P}^n = \bigsqcup_{\beta \in \mathcal{B}} S_\beta \]
into $G$-invariant locally closed subschemes and the morphism
\[ G \times_{P_\beta} Y_\beta^{ss} \rightarrow \mathbb{P}^n \]
induced by the group action
\[ G \times_{P_\beta} \mathbb{P}^n \rightarrow \mathbb{P}^n \]
is an isomorphism onto $S_\beta$.

To go from the stratification of the ambient projective space $\mathbb{P}^n$ to a stratification of $X$ we intersect the above stratification by taking fibre products. For any subscheme $S$ of $\mathbb{P}^n$ we let \[ S(X):= S \times_{\mathbb{P}^n} X \]
be the fibre product of $X$ and $S$ over $\mathbb{P}^n$. Then $\overline{Y}_\beta(X)$ is a closed subscheme of $X$ and $\hat{S}_\beta (X) = G \times_{P_\beta} \overline{Y}_\beta(X)$ is a projective completion of $S_\beta(X)$. The morphism
\[G \times_{P_\beta} Y_\beta^{ss}(X) \rightarrow X \]
is an isomorphism onto $S_\beta(X)$ by using the universal property of the fibre product $S_\beta(X)$ and the fact that $G \times_{P_\beta} Y_\beta^{ss}  \cong S_\beta$ for the ambient projective space. We have a stratification
\[ X= \bigsqcup_{\beta \in \mathcal{B}} S_\beta(X) \]
into $G$-invariant locally closed subschemes (although for some indices $\beta$ the stratum $S_\beta(X) $ may be empty). We note at this point that this stratification can be refined by taking connected components of $Z_\beta^{ss}(X)$ in the same way as it can for varieties (cf. Remark 3.1).

We can also define the canonical linearisation on $\hat{S}_\beta$ in exactly the same way as we do for varieties and this can be restricted to $\hat{S}_\beta(X)$. In this situation it is still true that the GIT quotient
\[ Z_\beta(X) /\!/_{\mathcal{L}_\beta} \mathrm{Stab} \beta \]
is a categorical quotient of the $G$-action on $S_\beta(X)$.

Finally we observe that the stratification $\{ S_\beta: \beta \in \mathcal{B} \}$ of $\mathbb{P}^n$ is unchanged (except for a minor modification of its
labelling) if we replace $\mathcal{O}_{\mathbb{P}^n}(1)$ with $\mathcal{O}_{\mathbb{P}^n}(m)$ for any $m>0$, and if we regard $\mathbb{P}^n$
as a $G$-invariant linear subspace of a bigger projective space $\mathbb{P}^N$ on
which $G$ acts linearly. Thus we obtain well defined constructions for any
projective scheme $X$ with an ample $G$-linearisation $\mathcal{L}$, which are
unaffected by replacing $\mathcal{L}$ with $\mathcal{L}^{\otimes m}$ for any
$m>0$. Moreover in the case when $X$ is a nonsingular projective variety these
constructions agree with those in $\S \S$ 2-3 (cf. Remark 2.3).

\section{Simpson's construction of moduli of semistable sheaves}\label{simp constr}\markright{simpson's construction of moduli of semistable sheaves} 

Let $W$ be a complex projective scheme with ample invertible sheaf $\mathcal{O}(1)$. We consider the moduli problem of classifying pure coherent algebraic sheaves on $W$ up to isomorphism. From now on we will use the term sheaf to mean coherent algebraic sheaf and unless otherwise specified sheaves will be on $W$. Gieseker introduced a notion of semistability for sheaves in \cite{gieseker_sheaves} and constructed coarse moduli spaces of semistable torsion free sheaves in the case when $W$ is a smooth projective variety of
dimension at most two. Maruyama generalised this to torsion free sheaves over integral projective schemes in \cite{maruyama1,maruyama2}. Later Simpson \cite{simpson} constructed coarse moduli spaces of semistable pure sheaves on an arbitrary complex projective scheme $W$. We follow the more general construction of Simpson where the moduli space of semistable pure sheaves on $W$ of fixed dimension and Hilbert polynomial is constructed as a GIT quotient of a subscheme $Q$ of a quot scheme by the action of a special linear group $G$. The linearisation is given by using Grothendieck's embedding of the quot scheme into a Grassmannian and then using the Pl\"ucker embedding of the Grassmannian into projective space.

We fix a rational polynomial $P \in \mathbb{Q}[x]$ of degree $e$ which takes integer values when $x$ is integral. Recall that a sheaf is pure of dimension $e$ if its support has dimension $e$ and all nonzero subsheaves have support of dimension $e$. Then we consider pure sheaves of dimension $e$ with Hilbert polynomial $P$ calculated with respect to $\mathcal{O}(1)$.

\begin{defn}
Let $\mathcal{F}$ be a pure sheaf of dimension $e$ over $W$. We define the {\it 
multiplicity} of $\mathcal{F}$ to be $r(\mathcal{F})=e! a_e$ where $a_e$ is the leading coefficient in the Hilbert polynomial of $e$. If $\mathcal{F}$ is torsion free this is just the rank of $\mathcal{F}$. The {\it reduced Hilbert polynomial} of $\mathcal{F}$ is defined to be the quotient $P(\mathcal{F})/r(\mathcal{F})$.
\end{defn}

\begin{defn}
A sheaf $\mathcal{F}$ is {\it semistable} if it is pure and every nonzero subsheaf $\mathcal{F}' \subset \mathcal{F}$ satisfies \[ \frac{P(\mathcal{F}')}{r(\mathcal{F}')} \leq \frac{P(\mathcal{F})}{r(\mathcal{F})}  \]
where the ordering on polynomials is given by lexicographic ordering of their coefficients. The sheaf is {\it stable} if the above inequality is strict for every proper nonzero subsheaf.
A semistable sheaf $\mathcal{F}$ has a Jordan--H\"{o}lder filtration
\[0 = \mathcal{F}_0 \subset \mathcal{F}_1 \subset \cdots \subset \mathcal{F}_s
= \mathcal{F}\]
where $\mathcal{F}_i/\mathcal{F}_{i-1}$ is stable with reduced Hilbert
polynomial $P(\mathcal{F})/r(\mathcal{F})$ for $1 \leq i \leq s$. This filtration is not in general canonical but the associated graded sheaf
\[Gr^{JH}(\mathcal{F}) = \bigoplus_{i=1}^s \mathcal{F}_i/\mathcal{F}_{i-1}\]
is canonically associated to $\mathcal{F}$ (up to isomorphism). Two semistable sheaves $\mathcal{F}$ and $\mathcal{G}$ over $W$ are {\it S-equivalent} if 
$Gr^{JH}(\mathcal{F})$ and $Gr^{JH}(\mathcal{G})$ are isomorphic.
\end{defn}

 Simpson shows that the semistable sheaves with Hilbert polynomial $P$ are bounded (see \cite{simpson}, Theorem 1.1), and hence we can choose $n >\!> 0$ so that all such sheaves are $n$-regular. In particular this means that for any such sheaf $\mathcal{F}$ the evaluation map $H^0(\mathcal{F}(n)) \otimes \mathcal{O}(-n) \rightarrow \mathcal{F}$
is surjective and the higher cohomology of $\mathcal{F}(n)$ vanishes, i.e.
\[ H^{i}(\mathcal{F}(n)) = 0 \: \: \mathrm{for} \: i > 0, \]
so that $P(n) =P(\mathcal{F},n)=\mathrm{dim} H^0(\mathcal{F}(n))$. 

Let $V$ be a vector space of dimension $P(n)$. Then the evaluation map for $\mathcal{F}$ and a choice of isomorphism $H^0(\mathcal{F}(n)) \cong V$ determine a point $\rho  : V \otimes \mathcal{O}(-n) \rightarrow \mathcal{F}$ in the quot scheme
\[ \mathrm{Quot}(V \otimes \mathcal{O}(-n), P) \]
of quotients with Hilbert polynomial $P$ of the sheaf $V \otimes \mathcal{O}(-n)$
on $W$.
We consider the open subscheme $Q \subset \mathrm{Quot}(V \otimes \mathcal{O}(-n), P) $ consisting of quotients $\rho : V \otimes \mathcal{O}(-n) \rightarrow \mathcal{F}$ such that $\mathcal{F}$ is pure of dimension $e$ and 
the map on sections $H^0(\rho(n)):V \to H^0(\mathcal{F}(n))$ 
induced by $\rho$ tensored with the identity on $\mathcal{O}(n)$ is an isomorphism. The group $G:=\mathrm{SL}(V)$ acts on this quot scheme by acting on the vector space $V$, so that $g \cdot \rho$ is the composition
\begin{center}{$
\begin{diagram} \node{g \cdot \rho : V \otimes \mathcal{O}(-n)} \arrow{e,t}{g^{-1 }\cdot} \node{V \otimes \mathcal{O}(-n)} \arrow{e,t}{\rho} \node{\mathcal{F}} \end{diagram}
$}\end{center}
for $g \in G$ and $\rho : V \otimes \mathcal{O}(-n) \rightarrow \mathcal{F}$ in the quot scheme. The subscheme $Q$ is preserved by this action and the $G$-orbits correspond to isomorphism classes of sheaves. 

Simpson considers a linearisation of this action given by an equivariant embedding of the quot scheme $\mathrm{Quot}(V \otimes \mathcal{O}(-n),P)$ 
into a Grassmannian. Grothendieck showed that for $m >\!> n$ the morphism
\[ \mathrm{Quot}(V \otimes \mathcal{O}(-n), P)  \longrightarrow \mathrm{Gr}(V \otimes H^0(\mathcal{O}(m-n)), P(m)) \]
\[ \rho : V \otimes \mathcal{O}(-n) \rightarrow \mathcal{F} \longmapsto H^0(\rho(m)) : V \otimes H  \rightarrow H^0(\mathcal{F}(m)) \]
is an embedding, where $H:=H^0(\mathcal{O}(m-n))$ and
$ \mathrm{Gr}(V \otimes H^0(\mathcal{O}(m-n)), P(m))$ is the Grassmannian of
$P(m)$-dimensional quotients of the vector space $V \otimes H$.
The Pl\"ucker embedding 
\[ \mathrm{Gr}(V \otimes H, P(m))  \hookrightarrow \mathbb{P}((\wedge^{P(m)} (V \otimes H))^*) \]
\[ H^0(\rho(m)) \mapsto \wedge^{P(m)} H^0(\rho(m))  \]
then gives an embedding of $  \mathrm{Quot}(V \otimes \mathcal{O}(-n), P)   $
in the projective space $ \mathbb{P}((\wedge^{P(m)} (V \otimes H))^*)  $.
Let $\overline{Q}$ denote the closure of $Q$ in the quot scheme 
$  \mathrm{Quot}(V \otimes \mathcal{O}(-n), P)   $, let
 $\mathcal{U}$ be the restriction to $\overline{Q} \times W$ of the
universal quotient sheaf on the product of the quot scheme and $W$, and let $\pi_{\overline{Q}}$ and $ \pi_W$ be the projections from $\overline{Q} \times W $ to $\overline{Q}$ and $W$. Then since $m>\!>n$ the higher cohomology groups
$H^i(\mathcal{F}(m))$ for $i>0$ all vanish for $\rho : V \otimes \mathcal{O}(-n) \rightarrow \mathcal{F}$ in $\mathrm{Quot}(V \otimes \mathcal{O}(-n),P)$ and
\begin{equation} \label{label} \mathcal{L} = \mathrm{det}(\pi_{\overline{Q}*}(\mathcal{U} \otimes \pi_W^*\mathcal{O}(m))) \end{equation} 
is the ample invertible sheaf corresponding to the embedding of $\overline{Q}$ into the projective space 
 $ \mathbb{P}((\wedge^{P(m)} (V \otimes H))^*)  $ above. There is a natural lift of the $G$-action on $\overline{Q}$ to the universal quotient $\mathcal{U}$ and this gives an action of $G$ on $\mathcal{L}$; by abuse of notation we let $\mathcal{L}$ denote this linearisation as well as the line bundle underlying it. We assume $n$ and $m$ are both chosen sufficiently large (for details see \cite{simpson}).

\begin{thm}(\cite{simpson}, Theorem 1.21) \label{simthm}
Let $W$ be a projective scheme, $e \leq \mathrm{dim}(W) $ a positive integer and $P$ a Hilbert polynomial of degree $e$. Then if $m>\!>n>\!>0$ the GIT quotient $\overline{Q} /\!/_{\mathcal{L}} G$ defined as above is a coarse moduli space for semistable sheaves of pure dimension $e$ with Hilbert polynomial $P$ up to
S-equivalence.
\end{thm}

\subsection{Calculating the Hilbert--Mumford function}\label{calc HM fn}

The Hilbert--Mumford criterion (see \cite{mumford}, Theorem 2.1) gives a way to test the (semi)stability of a point $\rho : V \otimes \mathcal{O}(-n) \rightarrow \mathcal{F}$ of $Q$ in terms of one-parameter subgroups of $G$. If $\lambda$ is a 1-PS then $\lim_{t\rightarrow 0} \lambda(t) \cdot \rho  \in \overline{Q}$ is a fixed point for the $\mathbb{C}^*$-action induced by $\lambda$, and so the group $\mathbb{C}^*$ acts on the fibre of $\mathcal{L}$ over this fixed point by some character of $\mathbb{C}^*$, say $t \mapsto t^w$ for some integer $w$. The Hilbert--Mumford function of $\rho : V \otimes \mathcal{O}(-n) \rightarrow \mathcal{F}$ evaluated at $\lambda$ is defined as
\[ \mu^\mathcal{L}(\rho, \lambda) := w .\]
Let
\[ M^{\mathcal{L}}(\rho)=\inf_{} \frac{\mu^\mathcal{L}(\rho, \lambda)}{||\lambda ||}, \]
where the infimum is taken over all non-trivial one-parameter subgroups $\lambda$ of $G$, and as before the norm is determined by an invariant inner product on the Lie algebra of the maximal compact subgroup $\mathrm{SU}(V)$ of $G=\mathrm{SL}(V)$. 
Then the Hilbert--Mumford criterion states that $\rho$ is semistable with respect to $\mathcal{L}$ if and only if 
$\mu^\mathcal{L}(\rho, \lambda) \geq 0$ for every 1-PS $\lambda$ of $G$,
or equivalently $M^{\mathcal{L}}(\rho) \geq 0$. If $\rho$ is unstable with respect to $\mathcal{L}$ then $M^{\mathcal{L}}(\rho)$ is negative and a non-divisible 1-PS achieving this value is said to be {adapted} to $\rho$
(cf. Remark \ref{KeN}).
In this section we will calculate the Hilbert--Mumford function $\mu^\mathcal{L}(\rho, \lambda)$ for any 1-PS $\lambda$  of $G$.

First of all we make use of the fact that any 1-PS  induces a decomposition of $V$
as a direct sum of weight spaces:
\[ \begin{array}{ccc} \left\{ \begin{array}{c}1 \mhyphen \mathrm{PSs} \:  \mathrm{ \: of \:} \mathrm{SL}(V)  \end{array} \right\} & \longleftrightarrow & \left\{ \begin{array}{c} \mathrm{decompositions \:} V= \bigoplus_{k \in \mathbb{Z}} V_k \\ \mathrm{such \: that \:} \sum k \mathrm{dim} V_k = 0 \end{array}  \right\} \\ & & \\ \lambda & \mapsto & V_k:=\{ v \in V : \lambda(t) \cdot v = t^k v  \}.  \end{array} \]
The relation $ \sum k \mathrm{dim} V_k = 0$
ensures that we get a 1-PS of the special linear group as opposed to the general linear group. Such a decomposition determines a filtration of $V$ given by
\[ \cdots \subseteq V_{\geq k+1} \subseteq V_{\geq k} \subseteq V_{\geq k-1} \subseteq \cdots \]
where $V_{ \geq k} := \oplus_{l\geq k} V_l$. 
There are only finitely many integers $k$ such that $V_k \neq 0$, say 
\[k_1 > \cdots > k_s; \]
let $V^{(i)} = V_{\geq k_i}$ for $i=1, \cdots s$. Then we obtain a map
\[ \begin{array}{ccc} \left\{ \begin{array}{c}1 \mhyphen \mathrm{PSs} \:  \mathrm{ \: of \:} \mathrm{SL}(V) \\ \lambda  \end{array} \right\} & \longrightarrow & \left\{ \begin{array}{c}  0=V^{(0)} \subset V^{(1)} \subset \cdots \subset V^{(s)}=V \\  \mathrm{filtrations \: of \:} V \: \mathrm{and \: integers \:} k_1 > \cdots > k_s\\ \mathrm{such \: that \:} \sum k_i \mathrm{dim} V^{(i)}/V^{(i-1)} = 0 \end{array}  \right\} \\  \end{array} \]
Let $\lambda$ be a 1-PS of $G=\mathrm{SL}(V)$ and let  $\rho : V \otimes \mathcal{O}(-n) \rightarrow \mathcal{F}$  be a point in $\overline{Q}$. Then the filtration of $V$ determined by $\lambda$ induces a filtration of $\mathcal{F}$ given by
\[ 0=\mathcal{F}^{(0)} \subset \mathcal{F}^{(1)} \subset \cdots \subset \mathcal{F}^{(s)}=\mathcal{F}\]
where $ \mathcal{F}^{(i)}= \rho (V^{(i)} \otimes \mathcal{O}(-n) ) $. 
Let $\mathcal{F}_{\geq k}$ denote the image of $V_{\geq k} \otimes \mathcal{O}(-n)$ under $\rho$ for any integer $k$. Then $\rho$ induces
\[ \rho_{k} : V_{k} \otimes \mathcal{O}(-n) \rightarrow \mathcal{F}_{k}= \mathcal{F}_{\geq k}/ \mathcal{F}_{\geq k+1} \] 
for each integer $k$; here $\mathcal{F}_k$ and $\rho_k$ can only be nonzero if $k=k_i$ for some $i$ with  $1 \leq i \leq s$. We define
\begin{equation} \label{rhobar} \overline{\rho} = \bigoplus_{k \in \mathbb{Z}}\rho_k :  \bigoplus_{k \in \mathbb{Z}}V_k  \otimes \mathcal{O}(-n)
\rightarrow \overline{\mathcal{F}}= \bigoplus_{k \in \mathbb{Z}} \mathcal{F}_k 
\end{equation}
(cf. \cite{huybrechts} $\S$4.4).
We now have a formula for the Hilbert--Mumford function.

\begin{lemma}\label{hmf}
The Hilbert--Mumford function evaluated at a one-parameter subgroup
$\lambda$ of $G=\mathrm{SL}(V)$ 
for a point $\rho: V \otimes \mathcal{O}(-n)
\to \mathcal{F}$ in $\overline{Q}$  is given by
\[ \mu^\mathcal{L}(\rho, \lambda)  = \sum_{i=1}^{s-1} (k_i - k_{i+1})\left(P(\mathcal{F}^{(i)}, m) - \mathrm{dim} V^{(i)}\frac{P(\mathcal{F},m)}{P(\mathcal{F},n)}\right) \]
where $V^{(i)}$ and $\mathcal{F}^{(i)}$ are defined as above.
\end{lemma}
\begin{proof} By \cite{huybrechts} Lemma 4.4.3 the fixed point $\lim_{t \rightarrow 0} \lambda(t) \cdot \rho$ in $\overline{Q}$
is equal to $ \overline{\rho}$. To calculate the value of the Hilbert--Mumford function we need to calculate the weight of the $\mathbb{C}^*$-action on the fibre at  ${\overline{\rho}}$ of the line bundle $\mathcal{L}$ defined at (\ref{label}). For this we follow the argument of \cite{huybrechts} Lemma 4.4.4, 
though using a left action as opposed to a right action. 
Since $m>\!>n>\!>0$ we have $H^i(\overline{\mathcal{F}}(m))=0$ for $i>0$ and
the line bundle 
\[\mathcal{L} = \mathrm{det}(\pi_{\overline{Q}*}(\mathcal{U} \otimes \pi_W^*\mathcal{O}(m)))
\]
has fibre
\[\mathrm{det} ( H^0(\overline{\mathcal{F}}(m)))^* =
\wedge^{P(m)} H^0(\overline{\mathcal{F}}(m))^* \]
at $\overline{\rho}$. The $\mathbb{C}^*$-action induced by $\lambda$ on $\rho_k$ has weight $-k$ because $\lambda(t) \cdot \rho_k$ is the composition
\begin{center}{$
\begin{diagram}  \node{ V_k \otimes \mathcal{O}(-n)} \arrow{e,t}{\lambda^{-1}(t) } \node{V_k \otimes \mathcal{O}(-n)} \arrow{e,t}{\rho_k} \node{\mathcal{F}_k} \end{diagram}
$}\end{center}
and
\[ \lambda^{-1}(t) \cdot v_k = t^{-k} v_k \: \mathrm{for \: all \:} v_k \in V_k. \]
Therefore the weight of the $\mathbb{C}^*$-action on $\mathrm{det}H^0(\mathcal{F}_k(m))$ is equal to $-k$ times the dimension of $H^0(\mathcal{F}_k(m))$, which is the value $P(\mathcal{F}_k,m)$ at $m$ of the Hilbert polynomial $P(\mathcal{F}_k)$. The weight of the $\mathbb{C}^*$-action on the fibre
of $\mathcal{L}$  over $\overline{\rho}$ is minus the sum of the weights of the $\mathbb{C}^*$-action on $\mathrm{det}H^0(\rho_k(m))$, and so
\[  \mu^\mathcal{L}(\rho, \lambda) = \sum_{k \in \mathbb{Z}} k P(\mathcal{F}_k, m) = \sum_{i=1}^s k_i P(\mathcal{F}_{k_i}, m) .\]
Since $\lambda$ is a 1-PS of the special linear group $G=\mathrm{SL}(V)$ we have $\sum_{i=1}^s k_i \mathrm{dim} V_{k_i} = 0$, so we may write this as
\[  \mu^\mathcal{L}(\rho, \lambda) = \sum_{i=1}^s  k_i \left(P(\mathcal{F}_{k_i},
 m) -  \mathrm{dim} V_{k_i} \frac{P(\mathcal{F},m)}{P(\mathcal{F},n)} \right)  \]
\[ = \sum_{i=1}^s  k_i \left(P(\mathcal{F}^{(i)},m) - P(\mathcal{F}^{(i+1)},
 m) -  \mathrm{dim} V^{(i)} \frac{P(\mathcal{F},m)}{P(\mathcal{F},n)}
+ \mathrm{dim} V^{(i+1)} \frac{P(\mathcal{F},m)}{P(\mathcal{F},n)} \right)\] 
\[ = k_s\left(P(\mathcal{F},
 m) -  \mathrm{dim} V \frac{P(\mathcal{F},m)}{P(\mathcal{F},n)}\right) +
\sum_{i=1}^{s-1}  (k_i - k_{i+1}) \left(P(\mathcal{F}^{(i)},m) -  \mathrm{dim} V^{(i)} \frac{P(\mathcal{F},m)}{P(\mathcal{F},n)} \right)
\]
which gives the required result since $ \mathrm{dim} V ={P(\mathcal{F},n)}$.
\end{proof} 

\section{The stratification of the closure of $Q$}\label{strat of Q}\markright{the stratification of the closure of $Q$}

We consider the group $G=\mathrm{SL}(V)$ acting on the subscheme $\overline{Q}$ 
 of the quot scheme
$\mathrm{Quot}(V \otimes \mathcal{O}(-n), P) $
with respect to the linearisation $\mathcal{L}$ defined at (\ref{label}), for which the GIT quotient 
$\overline{Q}/\!/_{\mathcal{L}} G$ is a coarse moduli space
for semistable sheaves on $W$ with Hilbert polynomial $P$. 
The linearisation $\mathcal{L}$ defines a $G$-equivariant embedding of $\overline{Q}$ in the projective space $ \mathbb{P}((\wedge^{P(m)} (V \otimes H))^*)  $ and we can choose a K\"{a}hler structure on 
$ \mathbb{P}((\wedge^{P(m)} (V \otimes H))^*)  $ which is invariant under the maximal compact subgroup $\mathrm{SU}(V)$ of $G=\mathrm{SL}(V)$. 
There is a stratification of this ambient projective space associated to this action and by intersecting this with $\overline{Q}$ we obtain a stratification
\[ \overline{Q} = \bigsqcup_{\beta \in \mathcal{B}} S_\beta\]
into G-invariant locally closed subschemes as in $\S$\ref{ext proj sch}. The aim of this
section is to prove Proposition \ref{strats agree} which relates the stratum
$S_\beta$ containing a point $\rho: V \otimes \mathcal{O}(-n)
\to \mathcal{F}$ of ${Q}$ to the Harder--Narasimhan type of the sheaf $\mathcal{F}$. Versions of this result have been well known for a long
time (cf. \cite{atiyahbott,kirwanarxiv,newstead_bundles} for the case when $W$ is
a nonsingular projective curve) but we provide a proof for the sake of
completeness.

Fix a basis of $V$ and pick the maximal torus $T \subset \mathrm{SU}(V)$ consisting of diagonal matrices of determinant 1 with entries in $S^1$. Then the Lie algebra of $T$ consists of purely imaginary tracefree diagonal matrices. We choose a positive Weyl chamber given by \[ \mathfrak{t}_+= \left\{ i \mathrm{diag}(a_1, \cdots , a_{\mathrm{dim}(V)}) : \begin{array}{c} a_i \in \mathbb{R} \mathrm{\: \: such \: that \:} \sum a_i =0   \\ \mathrm{and \:} a_1 \geq a_2 \geq \cdots \geq a_{\mathrm{dim}(V)} \end{array}\right\}.\]
The indexing set $\mathcal{B}$ for the stratification
$\{S_\beta:\beta \in \mathcal{B}\}$ is a finite set of points in $\mathfrak{t}_+$.
We note at this point that the strata $S_\beta$ may not be connected and so may be stratified further into their connected components (cf. Remark \ref{disconn}).

\subsection{The refined stratum associated to a fixed Harder--Narasimhan type}\label{refined strat}

Any sheaf of pure dimension $e$ over $W$ has a canonical filtration by subsheaves whose successive quotients are semistable with decreasing reduced Hilbert polynomials, known as the Harder--Narasimhan filtration. 

\begin{defn}
Let $\mathcal{F}$ be a pure sheaf; then its {\it Harder--Narasimhan filtration} is a filtration
\[ 0 = \mathcal{F}^{(0)} \subsetneq \mathcal{F}^{(1)} \subsetneq \cdots \subsetneq \mathcal{F}^{(s)}=\mathcal{F} \]
such that the successive quotients $\mathcal{F}^{(i)}/\mathcal{F}^{(i-1)}$ are semistable with decreasing reduced Hilbert polynomials
\[ \frac{P(\mathcal{F}^{(1)})}{r(\mathcal{F}^{(1)})} >  \frac{P(\mathcal{F}^{(2)}/\mathcal{F}^{(1)})}{r(\mathcal{F}^{(2)}/\mathcal{F}^{(1)})} > \cdots >  \frac{P(\mathcal{F}^{(s)}/\mathcal{F}^{(s-1)})}{r(\mathcal{F}^{(s)}/\mathcal{F}^{(s-1)})} .\]
We will denote by $\mathrm{Gr}^{HN}(\mathcal{F})$ the associated
graded sheaf
\[\mathrm{Gr}^{HN}(\mathcal{F}) = \bigoplus_{i=1}^s \mathcal{F}^{(i)}/\mathcal{F}^{(i-1)}.\]
We call the first sheaf $\mathcal{F}^{(1)}$ appearing in the Harder--Narasimhan filtration the {\it maximal destabilising subsheaf}. The Harder--Narasimhan type of $\mathcal{F}$ is specified by the vector of Hilbert polynomials of the successive quotients,
\[ HN(\mathcal{F}):=(P(\mathcal{F}^{(1)}),P(\mathcal{F}^{(2)}/\mathcal{F}^{(1)}), \cdots , P(\mathcal{F}^{(s)}/\mathcal{F}^{(s-1)})). \]
For each point $\rho : V \otimes \mathcal{O}(-n) \rightarrow \mathcal{F}$ in $Q$ we define the Harder--Narasimhan type of $\rho$ to be the Harder--Narasimhan type of $\mathcal{F}$. \end{defn}

The different types of Harder--Narasimhan filtrations allow us to decompose $Q$ 
into subsets of fixed Harder--Narasimhan type.

\begin{defn}
If $\tau$ is a Harder--Narasimhan type, let $R_\tau \subseteq Q$ be the set of points $\rho : V \otimes \mathcal{O}(-n) \rightarrow \mathcal{F}$ such that $\mathcal{F}$ has Harder--Narasimhan type $\tau$.
Then we can write $Q$ as
\[ Q = \bigsqcup_\tau R_\tau. \]
Let $\tau_0=(P)$ denote the trivial Harder--Narasimhan type; then $R_{\tau_0}$ parameterises semistable sheaves and so is equal to the stratum $S_0$
by  \cite{simpson} Theorem 1.21 (cf. Theorem \ref{simthm}).\end{defn}

For the rest of this section we fix a nontrivial Harder--Narasimhan type $\tau = (P_1, \dots , P_s)$, where $P_1,\ldots,P_s$ are polynomials of degree $e$
such that $P_1 + \cdots + P_s = P$,
and we assume that there is a sheaf of pure dimension $e$ over $W$
with this Harder--Narasimhan type. The following lemma shows that if $n$ is sufficiently large then $R_\tau$ parameterises all sheaves with Harder--Narasimhan type $\tau$.

\begin{lemma}\label{HN bdd}
The set of sheaves of pure dimension $e$ with Hilbert polynomial $P$ and Harder--Narasimhan type $\tau$ is bounded.
\end{lemma}
\begin{proof}
This follows from a result of Simpson (see \cite{simpson} Theorem 1.1) that
 a set of sheaves on $W$ of pure dimension $e$ and Hilbert polynomial $P$ is bounded if the slopes of their subsheaves are bounded above by a fixed constant,
where the slope of a sheaf is (up to multiplication by a positive constant) the second to top coefficient of its reduced Hilbert polynomial.
 Any sheaf $\mathcal{F}$ with Harder--Narasimhan type $\tau$ has a maximal destabilising subsheaf $\mathcal{F}^{(1)}$ with Hilbert polynomial $P_1$, and
all subsheaves of $\mathcal{F}$ have reduced Hilbert polynomial less than or equal to the reduced Hilbert polynomial of $\mathcal{F}^{(1)}$.  Let $\mu_1$ denote the slope of $\mathcal{F}^{(1)}$, which
depends only on the polynomial $P_1$; then any subsheaf $\mathcal{F}'$ of $\mathcal{F}$ has slope less than or equal to $\mu_1$ and this proves the result. 
\end{proof}

This boundedness result means that we may assume $n$ is chosen so that 
all pure sheaves with Hilbert polynomial $P$ and Harder--Narasimhan type $\tau$ are $n$-regular, and therefore are parameterised by $Q$. We may also assume that 
 all  sheaves with Harder--Narasimhan type $(P_{i_1},\ldots,P_{i_k})$ for 
any $1 \leq i_1 < i_2 < \cdots < i_k  \leq s$ are $n$-regular; in
particular the sheaves $\mathcal{F}^{(i)}$ occurring in the Harder--Narasimhan filtration of any sheaf $\mathcal{F}$ of Harder--Narasimhan type $\tau$ are 
 $n$-regular.

We want to show that the subset $R_\tau$ indexed by a fixed Harder--Narasimhan type is contained in a stratum $S_{\beta(\tau)}$ occurring in the stratification $\{S_\beta: \beta \in \mathcal{B}\}$. In order to do this we look for a candidate for $\beta=\beta(\tau)$ depending only on the information coming from the Harder--Narasimhan type $\tau$. The definitions of $Z_\beta$ and $Y_\beta$,
$Z_\beta^{ss}$ and $Y_\beta^{ss}$ are valid for any $\beta \in \mathfrak{t}$
and do not require $\beta$ to belong to the indexing set $\mathcal{B}$, but $Y_\beta^{ss}$ will only be nonempty when $\beta \in \mathcal{B}$. Therefore we can look for a candidate $\beta \in \mathfrak{t}_+$, and if $Y_\beta^{ss}$ is nonempty then this will imply that $\beta $ belongs to $\mathcal{B}$.

We fix a point $\rho : V \otimes \mathcal{O}(-n) \rightarrow \mathcal{F}$ in $R_\tau$ and let 
\[ 0 = \mathcal{F}^{(0)} \subsetneq \mathcal{F}^{(1)} \subsetneq \cdots \subsetneq \mathcal{F}^{(s)}=\mathcal{F} \]
denote the Harder--Narasimhan filtration of $\mathcal{F}$. We want to find  $\beta$ such that $\rho$ belongs to $Y_\beta$, so first we look for a 1-PS $\lambda$ of $G=\mathrm{SL}(V)$ which is adapted to $\rho$ (cf. Remark \ref{KeN}). We have seen that all 1-PSs give rise to filtrations of $\mathcal{F}$ and it is reasonable to expect that a 1-PS adapted to $\rho$ will give rise to the filtration of $\mathcal{F}$ which is most responsible for its instability, namely its Harder--Narasimhan filtration. With this in mind we let \[V^{(i)} :=   H^0(\rho(n))^{-1}(H^0(\mathcal{F}^{(i)}(n)))\] and choose a basis of $V$ 
(and corresponding maximal torus of $G=\mathrm{SL}(V)$) by 
first taking a basis of $V^{(1)}$, then extending to $V^{(2)}$ and so on. 
This gives us a decomposition
\[V = V_1 \oplus \cdots \oplus V_s\]
of $V$ such that $V^{(i)} = V_1 \oplus \cdots \oplus V_i$ and so
$V^{(i)}/V^{(i-1)} \cong V_i$.
Then we consider 1-PSs in $G=\mathrm{SL}(V)$ of the form
\[ \lambda(t) = \left( \begin{array}{cccc} t^{\beta_1}I_{V_{1}} & & & \\ &  t^{\beta_2}I_{V_{2}} & & \\ & & \ddots & \\ & & &  t^{\beta_s}I_{V_{s}} \end{array} \right) \]
where 
$\beta_1, \ldots , \beta_s$ are integers such that $\beta_1 > \cdots >\beta_s$ and $\sum \beta_i P(\mathcal{F}^{(i)}/\mathcal{F}^{(i-1)},n) = 0$. 

\begin{rmk}
Since we are assuming that each $\mathcal{F}^{(i)}$ and $\mathcal{F}^{(i)}/\mathcal{F}^{(i-1)}$ is $n$-regular, we have
 that 
\[P(\mathcal{F}^{(i)}/\mathcal{F}^{(i-1)},n) = \mathrm{dim}V^{(i)}/V^{(i-1)}
= \mathrm{dim}V_i\]
for each $i$.
\end{rmk}

Recall that a non-trivial 1-PS $\lambda$ of $G$ is adapted to $\rho$ if
\[ \frac{ \mu^\mathcal{L}(\rho, \lambda) }{ ||\lambda|| } \]
is minimal among non-trivial 1-PSs of $G$.
Therefore let us choose the integers $(\beta_1, \cdots , \beta_s)$ to minimise the function
\[ f(\beta_1, \cdots \beta_s) :=\frac{\sum_{i=1}^{s-1}(\beta_i - \beta_{i+1})\left(P(\mathcal{F}^{(i)},m) -  P(\mathcal{F}^{(i)},n) \frac{P(\mathcal{F},m)}{P(\mathcal{F},n)} \right)}{(\sum_{i=1}^s \beta^2_iP(\mathcal{F}^{(i)}/\mathcal{F}^{(i-1)},n))^{1/2}} \]
subject to the condition that $g(\beta_1, \cdots \beta_s) := \sum_{i=1}^s \beta_iP(\mathcal{F}^{(i)}/\mathcal{F}^{(i-1)},n) = 0$. We introduce a Lagrangian multiplier $\eta$ and define
\[ \Lambda(\beta_1, \cdots , \beta_s, \eta) :=  f(\beta_1, \cdots \beta_s) - \eta g(\beta_1, \cdots \beta_s); \]
then we look for solutions to 
\begin{equation} \label{label2} \frac{\partial}{\partial \beta_j} \Lambda(\beta_1, \cdots , \beta_s, \eta) = 0 \: \mathrm{for} \: j = 1, \cdots , s \: \mathrm{and} \: \frac{\partial}{\partial \eta} \Lambda(\beta_1, \cdots , \beta_s, \eta) = 0 .\end{equation} 
Note that for any $a \in \mathbb{R}_{>0}$, we have $f(a\beta_1, \cdots , a\beta_s) = f(\beta_1, \cdots , \beta_s)$ and $g(a\beta_1, \cdots ,a\beta_s) = 0$ is equivalent to $g(\beta_1, \cdots , \beta_s) =0$. 
It is easy to check that
\[ (\beta_1, \ldots, \beta_s, \eta) = \left(\frac{P(\mathcal{F},m)}{P(\mathcal{F},n)} - \frac{P_1(m)}{P_1(n)}, \cdots , \frac{P(\mathcal{F},m)}{P(\mathcal{F},n)} - \frac{P_s(m)}{P_s(n)}, 0 \right) \]
provides a solution to the equations (\ref{label2}). Note that
$\beta_1 > \cdots >\beta_s$ where 
\[ \beta_i = \frac{P(\mathcal{F},m)}{P(\mathcal{F},n)} - \frac{P_i(m)}{P_i(n)} \]
because the reduced Hilbert polynomial of $P_i$ is strictly greater than that of $P_{i+1}$. 
Now consider
\begin{equation} \label{defbeta}
\beta =  i \mathrm{diag}(\beta_1, \cdots, \beta_1, \beta_2, \cdots , \beta_2, \cdots , \beta_s \cdots \beta_s)  \in \mathfrak{t}_+\end{equation}
where $\beta_i$ appears $P_i(n)$ times.

\begin{rmk}
This $\beta$ depends on the Harder--Narasimhan type $\tau$ (as well as on $n$
and $m$) and will be written as $\beta = \beta(\tau)$ if it is necessary to 
make this dependence explicit. We note that for two distinct Harder--Narasimhan types $\tau$ and $\tau'$, for all $n$ and $m$ sufficiently large the associated weights $\beta(\tau)$ and $\beta(\tau')$ will also be distinct.
\end{rmk}

Consider the subschemes $Z_\beta$ and $Y_\beta$ of $\overline{Q}$ defined as at (\ref{zedbeta}) and (\ref{ybeta}). 

\begin{lemma}\label{gr object is in Z}
Suppose $n>\!> 0$, then the point $\rho:V \otimes \mathcal{O}(-n) \to \mathcal{F}$ in $R_\tau$
belongs to $Y_\beta$ where $\beta=\beta(\tau)$.
\end{lemma}
\begin{proof} Our assumptions on $n$ imply that $\mathcal{F}$ and all the subquotients appearing in its Harder--Narasimhan filtration are $n$-regular.
The point $\rho$ belongs to $Y_\beta$ if and only if the limit point 
\[ \lim_{t \rightarrow 0} \lambda_{\beta}(t) \cdot \rho = \overline{\rho} \]
of its path of steepest descent under the function $\mu \cdot \beta$ belongs to $Z_\beta$. By \cite{huybrechts} Lemma 4.4.3
this limit point is 
\[ \overline{\rho} : V \otimes \mathcal{O}(-n) \rightarrow \mathrm{Gr}^{HN}(\mathcal{F}) = \bigoplus_{i=1}^s \mathcal{F}^{(i)}/ \mathcal{F}^{(i-1)} .\]
The weight of $ \lambda_{\beta}$ acting on a point lying over $\overline{\rho}$ is given by 
\[-\mu^{\mathcal{L}}(\rho, \lambda_{\beta}) = \sum \frac{P_i(m)^2}{P_i(n)}-\frac{P(m)^2}{P(n)} \]
which is equal to $||\lambda_{\beta}||^2 = ||\beta||^2$, and so $\overline{\rho} \in Z_\beta$ as required. 
\end{proof}

Recall that we have a decomposition $V=V_1 \oplus \cdots \oplus V_s$ into weight spaces for the 1-PS $\lambda_{\beta}$. 

\begin{lemma} \label{zedblemma}
Let $\rho:V \otimes \mathcal{O}(-n) \to \mathcal{F}$ be a point in $\overline{Q}$. Then $\rho$ is fixed by the 1-PS $\lambda_{\beta}$ if and only if $\mathcal{F}$ has 
a decomposition 
\[\mathcal{F} = \mathcal{F}_1 \oplus \cdots \oplus \mathcal{F}_s\]
and we also have a decomposition
\[\rho=\rho_1 \oplus \cdots \oplus \rho_s\]
where $\rho_i:V_i \otimes \mathcal{O}(-n) \to \mathcal{F}_i$ lies in the quot scheme $\mathrm{Quot}(V_i \otimes \mathcal{O}(-n),P(\mathcal{F}_i))$.
\end{lemma}

The fixed point locus of $\lambda_\beta(\mathbb{C}^*)$ acting on $\overline{Q}$ decomposes into components indexed by the tuple of Hilbert polynomials of the direct summands. Let 
$Q_i:=\mathrm{Quot}(V_i \otimes \mathcal{O}(-n),P_i)$ and consider 
\[ F = \left\{ q \in \mathrm{Quot}(V \otimes \mathcal{O}(-n),P) : q = \oplus_{i=1}^s q_i \: \mathrm{such \: that} \: q_i \in Q_i \right\} \cong Q_1 \times \cdots \times Q_s. \]

\begin{cor}\label{cor on conn comp}
The scheme $F \cap \overline{Q}$ is a union of connected components of $Z_\beta$.
\end{cor}
\begin{proof}
Clearly $F$ is a union of connected components of the fixed point locus of the one-parameter subgroup $\lambda_{\beta}$. By definition $Z_\beta$ is the connected components of the fixed point locus in $\overline{Q}$ on which $\lambda_{\beta}$ acts with weight $|| \beta||^2$. Let $q = \oplus_{i=1}^s q_i$ be a point in $F$ where $q_i : V_i \otimes \mathcal{O}(-n) \rightarrow \mathcal{E}_i$ is a quotient sheaf in $Q_i$. The Hilbert-Mumford function $\mu^{\mathcal{L}}(q, \lambda_\beta)$ is equal to minus the weight of he action of $\lambda_\beta$ on the fibre of $\mathcal{L}$ over $q$. By direct calculation we have
\[ || \beta||^2 = \sum_{i=1}^s \beta_i^2 P_i(n)= \sum_{i=1}^s \frac{P_i(m)^2}{P_i(n)} - \frac{P(m)^2}{P(n)} \]
and
\[\mu^{\mathcal{L}}(q, \lambda_\beta) = \sum_{i=1}^s \beta_i P(\mathcal{E}_i,m) = \sum_{i=1}^s \beta_i P_i(m) =  \frac{P(m)^2}{P(n)} - \sum_{i=1}^s \frac{P_i(m)^2}{P_i(n)} \]
so that $F \cap \overline{Q}$ is a union of connected components of $Z_\beta$.
\end{proof}

\begin{rmk} \label{components}
Recall from Remark \ref{disconn} that from the decomposition
$Z_\beta = \sqcup Z_{(\tau')}$ into disjoint closed subsets
we get similar decompositions $Y_\beta = \sqcup Y_{(\tau')}$
and $Y_\beta^{ss} = \sqcup Y_{(\tau')}^{ss}$ and
\[S_\beta = \sqcup GY_{(\tau')}^{ss} \cong  \sqcup G \times_{P_\beta} Y_{(\tau')}^{ss}  \]
where $Y_{(\tau')} = p_\beta^{-1}(Z_{(\tau')}) \subseteq Y_\beta$
and $Y_{(\tau')}^{ss} = p_\beta^{-1}(Z_{(\tau')}^{ss})$. Thus $GY_{(\tau)}^{ss}
\cong  \sqcup G \times_{P_\beta} Y_{(\tau')}^{ss} $
is a union of connected components of $S_\beta$. 
\end{rmk}

We want to show that $\rho$  belongs to $Y_\beta^{ss}$, which is equivalent to showing that $\overline{\rho} \in Z_\beta^{ss}$. Recall that the subscheme $Z_\beta$ is invariant under the subgroup of $\mathrm{SL}(V)$ which stabilises $\beta$,
\[ \mathrm{Stab} \beta = \left( \prod_{i=1}^s \mathrm{GL}(V_i ) \right) \cap \mathrm{SL}(V). \]
The original linearisation $\mathcal{L}$ restricts to a $\mathrm{Stab} \beta$ linearisation on $Z_\beta$ which we also denote by $\mathcal{L}$. Associated to $-\beta$ is a character
\[ \chi_{-\beta} : \mathrm{Stab} \beta \rightarrow \mathbb{C}^* \]
\[ (g_1, \cdots g_s) \mapsto \Pi_{i=1}^s \mathrm{det}g_i^{-\beta_i} \]
which we can use to twist the linearisation $\mathcal{L}$; we let $\mathcal{L}^{\chi_{-\beta}}$ denote this twisted linearisation on $Z_\beta$. By definition
\[ Z_\beta^{ss} := Z_\beta^{\mathrm{Stab} \beta -ss}(\mathcal{L}^{\chi_{-\beta}}) \]
is the open subscheme of $Z_\beta$ whose geometric points are semistable for this $\mathrm{Stab} \beta$ action.

Note that 
the centre of $\mathrm{Stab} \beta$ is 
\[ Z(\mathrm{Stab} \beta) = \{ (t_1, \ldots , t_s) \in (\mathbb{C}^* )^{s}
: \prod_{i=1}^s t_i^{P_i(n)} =1 \}.\]
Consider the subgroup
\[G' = \prod_{i=1}^s \mathrm{SL}(V_i )\]
of $\mathrm{Stab} \beta$.

\begin{lemma}\label{lins equal} There is an isomorphism
$\mathrm{Stab}{\beta} \cong (G' \times Z(\mathrm{Stab} \beta) 
 )/ (\prod_{i=1}^s \mathbb{Z} / P_i(n) \mathbb{Z})$.
Furthermore, the semistable subscheme $F^{ss} := F^{\mathrm{Stab} \beta -ss}(\mathcal{L}^{\chi_{-\beta}})$ for the $\mathrm{Stab} \beta$ action on $F$ with respect to $\mathcal{L}^{\chi_{-\beta}}$ is equal to the semistable subset for the $G'$-action on $F$ with respect to $\mathcal{L}$.
\end{lemma}
\begin{proof}
The stabiliser of $\beta$ is
\[ \mathrm{Stab} \beta = \left( \prod_{i=1}^s \mathrm{GL}(V_i ) \right) \cap \mathrm{SL}(V) \] and there is a surjection 
\[ G' \times Z(\mathrm{Stab} \beta) \to \mathrm{Stab} \beta \]
\[ ((g_1', \dots, g_m'),  (t_1, \dots , t_s) ) \mapsto (t_1g_1', \dots, t_sg_s') \]
with kernel $\prod_{i=1}^s \mathbb{Z} / P_i(n) \mathbb{Z}$. Hence $\mathrm{Stab} \beta$ is the quotient of the 
 product $G' \times Z(\mathrm{Stab} \beta)$ by this product of the finite cyclic groups of order $P_i(n)$. However finite groups do not make any difference to GIT semistability,
so we can just consider the action of $G' \times Z(\mathrm{Stab} \beta)$. 

The centre $Z(\mathrm{Stab} \beta)$ fixes each point $q = \oplus q_i$ in $F$ and acts on the fibre of $\mathcal{L}$ at $q$ as multiplication by a character $\chi$. Since $q_i$ is multiplied by $t_i^{-1}$, $\mathrm{det}H^0(q_i(m))$ is multiplied by $t_i^{-P_i(m)}$, and we find that
$\chi(t_1, \cdots t_s)= \prod_{i=1}^s t_i^{P_i(m)}$. Since $\prod t_i^{P_i(n)} =1$ we may rewrite this as \[ \chi(t_1, \cdots t_s)= \prod_{i=1}^s t_i^{-\left(P_i(n) \frac{P(n)}{P(m)}-P_i(m)\right)}= \prod_{i=1}^s t_i^{-\beta_iP_i(n)}. \]
The centre acts on $\mathcal{L}_q$ via the character $\chi_{-\beta}$ and so it acts trivially on the fibre over the modified linearisation $\mathcal{L}^{\chi_{-\beta}}$. In particular, the semistable set for the action of $\mathrm{Stab} \beta = G' Z(\mathrm{Stab} \beta)$ with respect to $\mathcal{L}^{\chi_{-\beta}}$ is equal to the semistable set for the $G'$ action with respect to $\mathcal{L}$.
\end{proof}

Recall the following standard result:

\begin{lemma}\label{lemma GIT 2}
Let $X_1, \cdots , X_k$ be complex projective schemes and suppose $G_i$ is a reductive group acting on $X_i$ for $1 \leq i \leq k$. Let $\mathcal{L}_i$ be an ample linearisation of the $G_i$ action on $X_i$. Then
\[ (\prod_{i=1}^k X_i)^{\prod G_i -ss} (\bigotimes_{i=1}^k \pi_i^*\mathcal{L}_i)  = \prod_{i=1}^k X_i^{G_i - ss } (\mathcal{L}_i) \]
where $\pi_j : \prod_{i=1}^k X_i \rightarrow X_j $ is the projection map.
\end{lemma}

Recall that
\[ F \cong Q_1 \times \dots \times Q_s \]
where $Q_i= \mathrm{Quot}(V_i \otimes \mathcal{O}(-n),P_i)$. Consider the linearisation $\mathcal{L}_i = \mathrm{det}(\pi_{Q_i*}(\mathcal{U}_i \otimes \pi_W^*\mathcal{O}(m)))$ of the $\mathrm{SL}(V_i)$-action on $Q_i$ where $\mathcal{U}_i$ is the universal quotient sheaf on this quot scheme. By \cite{simpson}, Theorem 1.19 provided $n$ and $m$ are sufficiently large the points of the semistable subscheme
\[ Q_i^{ss} :=Q_i^{\mathrm{SL}(V_i)-ss}(\mathcal{L}_i)\]
are quotient sheaves $q_i : V_i \otimes \mathcal{O}(-n) \rightarrow \mathcal{E}_i$ where $\mathcal{E}_i$ is Gieseker semistable.

\begin{prop} \label{beta is an index}
Under the isomorphism $F \cong Q_1 \times \dots \times Q_s$ the semistable part of $F$ with respect to $\mathcal{L}^{\chi_{-\beta}}$ is isomorphic to the product of the GIT semistable subschemes $Q_i^{ss}$:
\[ F^{ss} \cong Q_1^{ss}\times \cdots \times Q_s^{ss} .\]
Furthermore, for $n$ and $m$ sufficiently large the limit point $\overline{\rho} \in  F^{ss} \cap Q \subset Z_\beta^{ss}$ and so $\beta$ is an index in the stratification of $\overline{Q}$.
\end{prop}
\begin{proof}
By Lemma \ref{lins equal}, we have that $ F^{ss} := F^{\mathrm{Stab} \beta -ss}(\mathcal{L}^{\chi_{-\beta}})=F^{G'-ss}(\mathcal{L})$ where $G' =\prod_{i=1}^s \mathrm{SL}(V_i )$. If we can show that $\mathcal{L}|_F  \cong \otimes \pi_i^* \mathcal{L}_i$, then, by Lemma \ref{lemma GIT 2}, \[F^{G'-ss}(\mathcal{L})=F^{G'-ss}(\otimes \pi_i^* \mathcal{L}_i) = Q_1^{ss} \times \cdots \times Q_s^{ss}\] where $\pi_i : F \cong \prod_{j=1}^s Q_j \rightarrow Q_i$ is the $i$th projection map. Let $i : F \hookrightarrow \mathrm{Quot}(V \otimes \mathcal{O}(-n),P)$ and $j : F \times W \hookrightarrow \mathrm{Quot}(V \otimes \mathcal{O}(-n),P) \times W$ denote the inclusions. Then
\[ \begin{array}{cl} \mathcal{L}|_{F} & = i^* \mathrm{det}(\pi_{*}(\mathcal{U} \otimes \pi_W^* \mathcal{O}(n))) \\ & \cong \mathrm{det} \pi_{F *} j^*(\mathcal{U} \otimes \pi_W^* \mathcal{O}(n)) \end{array}\]
since the determinant commutes with pullbacks and $i$ is flat. The universal family $\mathcal{U}$ pulls back via the morphism $j : F \times W \hookrightarrow \mathrm{Quot}(V \otimes \mathcal{O}(-n),P) \times W $ to the family $ \oplus_{i=1}^s p_i^* \mathcal{U}_i$ parameterised by $F$, where $p_i : F \times W 
 \cong (\prod_{j=1}^s Q_j) \times W \rightarrow Q_i \times W$ is the obvious projection map. 
Thus
\[ \begin{array}{cl} \mathcal{L}|_{F} &  \cong \mathrm{det} \left(\bigoplus_{i=1}^s \pi_{F *}(p_i^*\mathcal{U}_i \otimes (\pi_W^{F \times W})^* \mathcal{O}(n))\right) 
\\ & \cong \bigotimes_{i=1}^s \mathrm{det}  \pi_{F *}p_i^* (\mathcal{U}_i \otimes (\pi_W^{Q_i \times W})^* \mathcal{O}(n))
\\ & \cong \bigotimes_{i=1}^s \mathrm{det} \pi_i^* \pi_{Q_i*} (\mathcal{U}_i \otimes (\pi_W^{Q_i \times W})^* \mathcal{O}(n))
 \cong \bigotimes_{i=1}^s \pi_i^*\mathcal{L}_i.
 \end{array} \] 

We have $\overline{\rho}= \oplus \rho_i$ where $\rho_i : V_i \otimes \mathcal{O}(-n) \rightarrow \mathcal{F}^{(i)}/\mathcal{F}^{(i-1)}$ is a quotient of
$ V_i \otimes \mathcal{O}(-n) $ such that $H^0(\rho_i(n))$ is an isomorphism and $\mathcal{F}^{(i)}/\mathcal{F}^{(i-1)}$ is a semistable sheaf. We pick $n$ and then $m$ sufficiently large as in \cite{simpson} so that GIT semistability of points in $Q_i$ with respect to $\mathcal{L}_i$ is equivalent to Gieseker semistability of the associated sheaves. Then
\[Q_i^{ss}:=Q_i^{\mathrm{SL}(V_i )-ss}(\mathcal{L}_i)\]
 is the open subset of quotients parameterising semistable sheaves. By definition of the Harder--Narasimhan filtration  $\rho_i
\in  Q_i^{ss}$ and so $\overline{\rho} \in Q \cap F^{ss} \subset Z_\beta^{ss}$. In particular $S_\beta$ is nonempty and so $\beta$ is an index for the stratification of $\overline{Q}$.
\end{proof}

\begin{prop}\label{strats agree} 
Choose an ordered basis of $V$ and a positive Weyl chamber $\mathfrak{t}_+$ in the Lie algebra of the associated maximal torus of $G=\mathrm{SL}(V)$. Let $\tau = (P_1, \dots P_s)$ be a Harder--Narasimhan type and let $\beta = \beta(\tau) = \beta(\tau,n,m) \in \mathfrak{t}_+$ be as at (\ref{defbeta}). If $n$ and $m$ are sufficiently large, then we can give $R_\tau$ a scheme structure such that every connected component of $R_\tau$ is a connected component of $S_\beta$.
\end{prop}
\begin{proof} Let $n$ and $m$ be chosen as in Proposition \ref{beta is an index}. Let $R_i$ be the open subscheme of $Q_i$ consisting of quotient sheaves $ q_i : V_i \otimes \mathcal{O}(-n) \rightarrow \mathcal{E}_i$ which are pure of dimension $e$ and such that $H^0(q_i(n))$ is an isomorphism. Let $R_i^{ss}$ denote the semistable subscheme for the $\mathrm{SL}(V_i)$-action on $R_i$. Then consider the subschemes
\[ Z_{(\tau)}^{ss} = \{ q = \oplus_{i=1}^s q_i :  (q_i : V_i \otimes \mathcal{O}(-n) \rightarrow \mathcal{E}_i) \in R_i^{ss} \} \]
of $Z_\beta^{ss}$ and $Y_{(\tau)}^{ss}=p_\beta^{-1}(Z_{(\tau)}^{ss})$ of $Y_\beta^{ss}$.

Any quotient sheaf $q: V \otimes \mathcal{O}(-n) \rightarrow \mathcal{F}$ in $Y_{(\tau)}^{ss}$ has a filtration and associated graded object $\overline{q}: V \otimes \mathcal{O}(-n) \rightarrow \mathcal{F}$ for which the successive quotients are semistable with Hilbert polynomials $P_1, \dots, P_s$; i.e. $\mathcal{F}$ has Harder--Narasimhan type $\tau$. As
$R_\tau$ is $G$-invariant it follows immediately that every point in $GY_{(\tau)}^{ss}$ is a point in $R_\tau$. Conversely let $\rho: V \otimes \mathcal{O}(-n) \rightarrow \mathcal{E}$ be any point in $R_\tau$; then the Harder--Narasimhan filtration of $\mathcal{E}$ gives rise to a filtration of $V$ by subspaces $W^{(i)}= H^0(q(n))^{-1}(H^0(\mathcal{E}^{(i)}(n)))$. We choose $g \in G=\mathrm{SL}(V)$ to be a change of basis matrix sending $W^{(i)}$ to $V^{(i)}$ for each $i$, which
is possible since $\dim W^{(i)} = \dim V^{(i)} = \sum_{j\leq i} P_i(n)$. Then $g \cdot q \in Y_{(\tau)}^{ss}$ 
by Proposition \ref{beta is an index}, so 
\[R_\tau = GY_{(\tau)}^{ss} \cong G \times_{P_\beta} Y_{(\tau)}^{ss}\] and this gives the set $R_\tau$ its scheme structure. 

Since $\overline{R}_i^{ss} = R_i^{ss}$ (cf. \cite{simpson} Theorem 1.19) the subscheme $Z_{(\tau)}^{ss}$ is closed in $F^{ss} \cap \overline{Q}$ and is thus a union of connected components of $Z_\beta^{ss}$ by Corollary \ref{cor on conn comp}. It follows that $R_\tau= GY_{(\tau)}^{ss}$ is a union of connected components of $S_\beta$ by Remark \ref{components}.
\end{proof}

\section{$n$-rigidified sheaves of fixed Harder--Narasimhan type}\label{nrigidified}\markright{$n$-rigidified sheaves of fixed harder--narasimhan type} 

As in the previous section we let $\tau = (P_1,  \dots P_s)$ be a Harder--Narasimhan type and let $\beta = \beta(\tau) = \beta(\tau,n,m) \in \mathfrak{t}_+$ be the associated rational weight given at (\ref{defbeta}). In section \ref{moduli unstable sheaves} below we consider the action of $\mathrm{Stab} \beta$ on the closure $\overline{Y}_{(\tau)}$ in
the quot scheme $\mathrm{Quot}(V \otimes \mathcal{O}(-n),P)$ of the subscheme $Y_{(\tau)}^{ss}$ defined in the proof of Proposition \ref{strats agree}. We know the $P_\beta$-orbits in $Y_{(\tau)}^{ss}$ correspond to 
$G$-orbits in $R_\tau \cong G  \times_{P_\beta} Y_{(\tau)}^{ss}$ and thus to 
isomorphism classes of sheaves of Harder--Narasimhan type $\tau$, and so in this section we study the objects parametrised by the $\mathrm{Stab} \beta$-orbits 
in $Y_{(\tau)}^{ss}$.

\begin{defn}\label{n-rig def}
Let $n$ be a positive integer and $\mathcal{F}$ be a sheaf with Harder--Narasimhan type $\tau$. Let $0 \subset \mathcal{F}^{(1)} \subset \cdots \subset \mathcal{F}^{(s)}= \mathcal{F}$ denote the Harder--Narasimhan filtration of $\mathcal{F}$ and $\mathcal{F}_i:= \mathcal{F}^{(i)}/ \mathcal{F}^{(i-1)}$ denote the successive quotients. Then an $n$-rigidification for $\mathcal{F}$ is an isomorphism
\[ H^0( \mathcal{F}(n)) \cong  \oplus_{i=1}^s H^0(\mathcal{F}_i(n))\]
which is compatible with the inclusion morphisms $j^{(i)} : \mathcal{F}^{(i)} \hookrightarrow \mathcal{F}$ and projection morphisms $\pi^{(i)} : \mathcal{F}^{(i)} \rightarrow \mathcal{F}_i$; that is, for each $i$ we have a commutative triangle
\begin{center}{$
\begin{diagram} \node{H^0(\mathcal{F}^{(i)}(n))} \arrow{e,t}{j^{(i)}_*} \arrow{s,b}{\pi^{(i)}_*} \node{H^0(\mathcal{F}(n))}  \arrow{sw,b}{}  \\ \node{H^0(\mathcal{F}_i(n))} \end{diagram}
$}\end{center}
where the unlabelled arrow is the given isomorphism $H^0( \mathcal{F}(n)) \cong  \oplus_{i=1}^s H^0(\mathcal{F}_i(n))$ followed by the $i$th projection. An isomorphism of two $n$-rigidified sheaves $\mathcal{E}$ and $\mathcal{F}$ is an isomorphism of sheaves $\phi: \mathcal{E} \cong \mathcal{F}$ such that for each $i$  the induced isomorphisms $H^0(\mathcal{E}^{(i)}(n)) \cong H^0(\mathcal{F}^{(i)}(n))$ are compatible with the $n$-rigidifications; i.e., we have a commutative square of isomorphisms
\begin{center}{$
\begin{diagram} \node{H^0(\mathcal{E}(n))} \arrow{e,t}{} \arrow{s,b}{} \node{H^0(\mathcal{F}(n))}  \arrow{s,b}{}  \\ \node{\oplus_{i=1}^s H^0(\mathcal{E}_i(n))} \arrow{e,t}{} \node{\oplus_{i=1}^s H^0(\mathcal{F}_i(n))} \end{diagram}
$}\end{center}
where the horizontal morphisms are induced by the isomorphism $\phi$ and the vertical morphisms are the given $n$-rigidifications for each sheaf.
\end{defn}

\begin{rmk}\label{n-rigs exist}
Any sheaf $\mathcal{F}$ with Harder--Narasimhan type $\tau$ has an $n$-rigidification for $n >\!> 0$ where $n$ is sufficiently large so the higher cohomology of $\mathcal{F}(n)$ and $\mathcal{F}_i(n)$ vanish. In fact if we pick $n$ as required for Proposition \ref{strats agree}, then the quotient sheaf $q : V \otimes \mathcal{O}(-n) \rightarrow \mathcal{F} $ has a natural $n$-rigidification coming from the eigenspace decomposition $V = \oplus_{i=1}^s V_i$ of $V$ for $\lambda_\beta(\mathbb{C}^*)$ and the isomorphisms $V \cong H^0(\mathcal{F}(n))$ and $V_i \cong H^0(\mathcal{F}_i(n))$ induced by $q$.
\end{rmk}

\begin{lemma}\label{n-rig lemma}
Consider the $n$-rigidified sheaves represented by points
$q : V \otimes \mathcal{O}(-n) \rightarrow \mathcal{E} $ and $q' : V \otimes \mathcal{O}(-n) \rightarrow \mathcal{F} $ in $Y_{(\tau)}^{ss}$ 
as in  Remark \ref{n-rigs exist}. These $n$-rigidified sheaves are isomorphic if and only if there is some $g \in \Pi_{i=1}^s \mathrm{GL}(V_i)$ such that $g \cdot q = q'$.
\end{lemma}
\begin{proof}
If $\mathcal{E}$ and $\mathcal{F}$ are isomorphic as $n$-rigidified sheaves then, in particular, they are isomorphic as sheaves and so there is some $g \in \mathrm{GL}(V)$ such that $g \cdot q = q'$. As $q$ and $q'$ are both in $Y_\beta^{ss}$ 
and $G Y_\beta^{ss} \cong G \times_{P_\beta} Y_\beta^{ss} $,
we know that $g \in P_\beta$ is block upper triangular with respect to the blocks for $\beta$. Then as the isomorphism is compatible with the $n$-rigidifications, we see that $g$ must be block diagonal; i.e., $g$ is an element of $\mathrm{Stab} \beta = \Pi_{i=1}^s \mathrm{GL}(V_i)$.

Conversely if there is a $g \in \Pi_{i=1}^s \mathrm{GL}(V_i)$ such that $g \cdot q = q'$ then this induces a sheaf isomorphism $\mathcal{E} \cong \mathcal{F}$. The fact that $g$ is block diagonal with respect to the blocks for $\beta$ means this isomorphism is an isomorphism of $n$-rigidified sheaves.
\end{proof}

\section{Moduli spaces of rigidified unstable sheaves}\label{moduli unstable sheaves}\markright{moduli spaces of rigidified unstable sheaves}

In this final section we construct moduli spaces of $n$-rigidified sheaves of fixed Harder--Narasimhan type $\tau$ as GIT quotients $\overline{Y_{(\tau)}}/\!/
\mathrm{Stab} \beta$, where $\beta = \beta(\tau)$, with respect to perturbations of the canonical linearisation $\mathcal{L}_\beta$ for the $\mathrm{Stab} \beta$-action on $\overline{Y_{(\tau)}}$.

\begin{rmk}
We would like to construct moduli spaces of sheaves of fixed
Harder--Narasimhan type $\tau$ as GIT quotients $\overline{Y_{(\tau)}}/\!/P_\beta$ or $G \times_{P_\beta} \overline{Y_{(\tau)}}
/\!/G$ for suitable perturbations of the linearisation $\mathcal{L}_\beta$.
However there are difficulties here since in general the group $P_\beta$ is 
not reductive and the linearisation $\mathcal{L}_\beta$ on 
$G \times_{P_\beta} \overline{Y_{(\tau)}}$ is not ample.
\end{rmk}

\begin{rmk}
 Moduli spaces of unstable bundles of rank 2 on the projective plane have been constructed by Str{\o}mme in \cite{stromme} and this has been generalised to sheaves with Harder--Narasimhan filtrations of length two over smooth projective varieties by Dr\'ezet in \cite{drezet}.
\end{rmk}

We will define a notion of $\theta$-(semi)stability for sheaves over $W$ of a fixed Harder--Narasimhan type $\tau$ corresponding to a sequence of
Hilbert polynomials $(P_1, \ldots, P_s)$
and a moduli functor of $\theta$-semistable $n$-rigidified sheaves of Harder--Narasimhan type $\tau$
over $W$. 
This notion of $\theta$-(semi)stability depends on a parameter $\theta \in \mathbb{Q}^s$ (see Definition  \ref{defthetastability} below), and we will show that if $m>\!> n >\!> 0$ then $\theta$ determines for us a perturbed $\mathrm{Stab} \beta$-linearisation on the closure $\overline{Y}_{(\tau)}$ of $Y_{(\tau)}$ as in $\S$\ref{how to perturb} with the following properties:

\medskip

(i) any $\rho:V \otimes \mathcal{O}(-n) \to \mathcal{F}$ in $Y_\tau^{ss}$ is GIT
semistable for the perturbed linearisation associated to $\theta$ if and only if 
the sheaf $\mathcal{F}$ of Harder--Narasimhan type $\tau$ is $\theta$-semistable
(Theorem \ref{ss agrees} below), and

\medskip

(ii) the associated GIT quotient is a projective scheme which corepresents the moduli functor of $\theta$-semistable $n$-rigidified sheaves of Harder--Narasimhan type $\tau$
over $W$ (Theorem \ref{thm6.16} below).

\medskip

Fix a Harder--Narasimhan type $\tau =( P_1, \cdots , P_s)$ and let $P= \sum_i P_i$; then, by Proposition \ref{strats agree}, for $n$ and $m$ sufficiently large the subvariety $R_\tau = GY_{(\tau)}^{ss} \cong G
\times_{P_\beta} Y_{(\tau)}^{ss}$ of $Q$ parametrising sheaves of Harder--Narasimhan type $\tau$
is a union of connected components of a stratum $S_{\beta(\tau)}$ in the stratification $\{S_\beta: \beta \in \mathcal{B}\}$ of $\overline{Q}$ given by
\[ \beta(\tau) = i \mathrm{diag}(\beta_1, \dots, \beta_1, \dots, \beta_s, \dots \beta_s) \in \mathfrak{t}_+\] 
where
\[ \beta_i= \frac{P(m)}{P(n)} - \frac{P_i(m)}{P_i(n)}\]
appears $P_i(n)$ times.
The stratum $S_\beta$ for $\beta = \beta(\tau)$ is isomorphic to  $G \times_{P_\beta} Y_\beta^{ss}$ and as in $\S$\ref{how to modify} we consider linearisations of the $G$-action on the projective completion
\[ \hat{S_\beta}:= G \times_{P_\beta} \overline{Y}_\beta,\]
where $ \overline{Y}_\beta$ is the closure of $Y_\beta^{ss}$ in $\overline{Q}$. Since $R_\tau \cong G \times_{P_\beta} Y_{(\tau)}^{ss}$ where $Y_{(\tau)}^{ss}$
is a union of connected components of $Y^{ss}_\beta$ we let 
\[\hat{R_\tau} = G \times_{P_\beta} \overline{Y}_{(\tau)}\]
where $\overline{Y}_{(\tau)}$ is the closure of $Y_{(\tau)}^{ss}$ in $\overline{Q}$; 
this is the closure of $R_\tau$ in $\hat{S_\beta}$ and is a projective completion of $R_\tau$.

Let $\mathcal{L}_\beta$ denote the canonical linearisation on $\hat{S_\beta}$ as defined in $\S$\ref{nat linear on hat var} and let $\mathcal{L}_\beta$ also denote its restriction to $\hat{R_\tau}$.
As was noted in $\S$\ref{how to modify}, ${S_\beta}$ and $R_\tau$ have categorical quotients
\[S_\beta \to Z_{\beta} /\!/_{\mathcal{L}_\beta} \mathrm{Stab}{\beta}\]
and
\[R_\tau \to Z_{(\tau)} /\!/_{\mathcal{L}_\beta} \mathrm{Stab}{\beta}\]
but these are far from orbit spaces: the map $p_\beta : Y_{(\tau)}^{ss} \rightarrow Z_{(\tau)}^{ss}$ sends a point $y$ to the graded object associated to its Harder--Narasimhan filtration and since $p_\beta(y)$ is contained in the orbit closure of $y$ these points are S-equivalent, in the sense that they
represent the same points in the categorical quotient. In fact two sheaves $\mathcal{F}$ and $\mathcal{G}$ with Harder--Narasimhan type $\tau$ are S-equivalent in this context if and only if the graded objects associated to their Jordan--H\"{o}lder filtrations are isomorphic. We would like a finer notion of equivalence.

As  was noted in $\S$\ref{how to perturb}, one possible approach to avoiding this problem is to perturb the canonical linearisation, but applying GIT to either the canonical $G$-linearisation on $\hat{S}_\beta$ or
the canonical $P_\beta$-linearisation on $\overline{Y}_\beta$ is delicate. 
So instead we will consider perturbations of the canonical $\mathrm{Stab} \beta$-linearisation $\mathcal{L}_\beta$ on $\overline{Y}_\beta$ given by making a small perturbation to the character $\chi_{-\beta} : \mathrm{Stab} \beta \rightarrow \mathbb{C}^*$ used to twist $\mathcal{L}$.

\subsection{Semistability} We will choose a perturbation of the canonical $\mathrm{Stab} \beta$-linearisation $\mathcal{L}_\beta$ on $\overline{Y}_{(\tau)}$ which depends on a parameter $\theta=(\theta_1, \dots , \theta_s) \in \mathbb{Q}^s$. 
A notion of (semi)stability with respect to this parameter $\theta$ will be defined for all sheaves over $W$ with Harder--Narasimhan type $\tau$. Before stating the definition we first need an easy lemma which enables us to write down the Harder--Narasimhan filtration of a direct sum of pure sheaves $\mathcal{E} \oplus \mathcal{F}$ in terms of the Harder--Narasimhan filtrations of  $\mathcal{E}$ and $\mathcal{F}$.

\begin{lemma} \label{HN direct sum}
Let $\mathcal{E}$ and $\mathcal{F}$ be pure sheaves of dimension $e$ with Harder--Narasimhan filtrations
\[ 0 \subset \mathcal{E}^{(1)} \subset \cdots \subset \mathcal{E}^{(N)} = \mathcal{E}\]
and
\[ 0 \subset \mathcal{F}^{(1)} \subset \cdots \subset \mathcal{F}^{(M)} = \mathcal{F}.\]
Then the maximal destabilising subsheaf of $\mathcal{E} \oplus \mathcal{F}$ is
\begin{enumerate}
\renewcommand{\labelenumi}{\roman{enumi})}
\item $ \mathcal{E}^{(1)} $ if $ P(\mathcal{E}^{(1)}) \: r( \mathcal{F}^{(1)}) > P(\mathcal{F}^{(1)}) \: r( \mathcal{E}^{(1)})$,
\item $\mathcal{F}^{(1)} $ if $P(\mathcal{E}^{(1)}) \: r( \mathcal{F}^{(1)} ) < P(\mathcal{F}^{(1)}) \: r( \mathcal{E}^{(1)})$,
\item $\mathcal{E}^{(1)}\oplus \mathcal{F}^{(1)} $ if $ P(\mathcal{E}^{(1)}) \: r( \mathcal{F}^{(1)}) = P(\mathcal{F}^{(1)}) \: r(\mathcal{E}^{(1)})$.
\end{enumerate}
\end{lemma}
\begin{proof}
Suppose $P(\mathcal{E}^{(1)}) \: r( \mathcal{F}^{(1)}) > P(\mathcal{F}^{(1)}) \: r( \mathcal{E}^{(1)}) $; then we need to show $\mathcal{E}^{(1)}$ is the maximal destabilising subsheaf of $\mathcal{E} \oplus \mathcal{F}$. We know $\mathcal{E}^{(1)}$ is semistable and we also claim that there is no sheaf $\mathcal{G} \subset \mathcal{E} \oplus \mathcal{F}$ with reduced Hilbert polynomial great than $\mathcal{E}^{(1)}$. To prove this suppose such a sheaf $\mathcal{G}$ exists; then we may assume without loss of generality that $\mathcal{G}$ is semistable. As $\mathrm{Hom}(\mathcal{G}, \mathcal{E}) = 0$, the composition
\[ \mathcal{G} \hookrightarrow \mathcal{E} \oplus \mathcal{F} \rightarrow \mathcal{E} \]
is zero and so $\mathcal{G}$ is contained completely in $\mathcal{F}$. This contradicts the fact that $\mathcal{F}^{(1)}$ is the maximal destabilising subsheaf in $\mathcal{F}$.

Now suppose there is $\mathcal{E}^{(1)} \subsetneq \mathcal{G} \subset \mathcal{E} \oplus \mathcal{F}$ such that $\mathcal{G}$ and $\mathcal{E}^{(1)}$ have the same reduced Hilbert polynomial. Then the composition
\[ \mathcal{G} \hookrightarrow \mathcal{E} \oplus \mathcal{F} \rightarrow \mathcal{F} \]
is zero and so $\mathcal{G} $ is contained in $\mathcal{E}$ which contradicts the fact that $\mathcal{E}^{(1)}$ is the maximal destabilising subsheaf in $\mathcal{E}$. Therefore $\mathcal{E}^{(1)}$ is the maximal destabilising subsheaf of the direct sum.

The other cases follow from similar standard arguments and will be omitted.
\end{proof}

\begin{defn}
We say a sheaf $\mathcal{F}$ is $\tau$-compatible if it has a filtration 
\[ 0 \subseteq \mathcal{F}^{(1)} \subseteq \cdots \subseteq \mathcal{F}^{(s)}=\mathcal{F}\]
such that $\mathcal{F}_i = \mathcal{F}^{(i)}  / \mathcal{F}^{(i-1)} $, if nonzero, is semistable with reduced Hilbert polynomial $P_i/r_i$
where $\tau = (P_1,\ldots,P_s)$. We call such a filtration a generalised Harder--Narasimhan filtration of $\mathcal{F}$; it is the same as the Harder--Narasimhan filtration of $\mathcal{F}$ except that we may have
$\mathcal{F}^{(i)}  = \mathcal{F}^{(i-1)} $ for some $i$. Note that the generalised Harder--Narasimhan filtration of a $\tau$-compatible sheaf $\mathcal{F}$ is uniquely determined by $\mathcal{F}$ and $\tau$.
\end{defn}

Of course any sheaf of Harder--Narasimhan type $\tau$ is $\tau$-compatible.

\begin{defn} \label{defthetastability}
A $\tau$-compatible sheaf $\mathcal{F}$ is $\theta$-semistable if for all proper nonzero $\tau$-compatible subsheaves $\mathcal{F}' \subset \mathcal{F}$ for which $\mathcal{F}/\mathcal{F}'$ is also $\tau$-compatible we have
\[ \frac{\sum_{i=1}^s \theta_i P(\mathcal{F}'_i)}{P(\mathcal{F}')} \geq \frac{\sum_{i=1}^s \theta_i P(\mathcal{F}_i)}{P(\mathcal{F})} \]
where $\mathcal{F}'_i$ and $\mathcal{F}_i$ denote the successive quotients appearing in the generalised Harder--Narasimhan filtrations of  $\mathcal{F}'$
 and $\mathcal{F}$. We say $\mathcal{F}$ is $\theta$-stable if this inequality is strict for all such subsheaves.
\end{defn}

\begin{rmk}\label{assum on theta}
To get a nontrivial notion of semistability we will always assume that the $\theta_i$ are not all equal to each other. In addition, we will usually assume that for all $m >\!> n >\!> 0$
\begin{equation} \frac{\sum \theta_i P_i(n)}{P(n)} \geq \frac{\sum \theta_i P_i(m)}{P(m)}.\label{cross} \end{equation}
If (\ref{cross}) does not hold we can still define $\theta$-(semi)stability but
there will be no $\theta$-semistable sheaves with Harder--Narasimhan type $\tau$.
\end{rmk}

\subsection{Families and the moduli functor}

Let $S$ be a complex scheme, and recall that a flat family of sheaves over $W$ parametrised by $S$ is a sheaf $\mathcal{V}$ over $W \times S$ which is flat over $S$. We say this is a flat family of semistable sheaves which are pure of dimension $e$ with Hilbert polynomial $P$ if for each point $s \in S$ the sheaf $\mathcal{V}_s:= \mathcal{V}|_{W \times \{ s \}}$ is   a semistable pure 
sheaf of dimension $e$ with Hilbert polynomial $P$. We say two flat families $\mathcal{V}$ and $\mathcal{W}$ over $W$ parametrised by $S$ are isomorphic if there is a line bundle $L$ on $S$ such that $\mathcal{V} \cong \mathcal{W} \otimes \pi_S^*L$
where $\pi_S:W \times S \to S$ is the projection. Given a morphism $f: T \rightarrow S$ we can pull back a family on $S$ to a family on $T$ in the standard way.

\begin{defn}
Let $\tau=(P_1, \dots, P_s)$ be a Harder--Narasimhan type of a pure sheaf of dimension $e$. A flat family $\mathcal{V}$ of sheaves over $W$ parametrised by $S$ has Harder--Narasimhan type $\tau$ if $\mathcal{V}$ is a family of pure sheaves of dimension $e$ with Hilbert polynomial $\sum_{i=1}^m P_i$ and there is a filtration by subsheaves
\[ 0 \subsetneq \mathcal{V}^{(1)} \subsetneq \dots \subsetneq \mathcal{V}^{(s)}= \mathcal{V} \]
such that $\mathcal{V}_i =\mathcal{V}^{(i)}/\mathcal{V}^{(i-1)}$ is a flat family of semistable sheaves of pure of dimension $e$ with Hilbert polynomial $P_i$.

Let $n$ be a positive integer. A flat family $\mathcal{V}$ of $n$-rigidified sheaves of Harder--Narasimhan type $\tau$ over $W$ parametrised by $S$ is a flat family $\mathcal{V}$ of sheaves of Harder--Narasimhan type $\tau$ parametrised by $S$ which has an $n$-rigidification; i.e., an isomorphism
\[ H^0(\mathcal{V}(n)) \cong \oplus_{i=1}^s H^0(\mathcal{V}_i(n)) \]
which is compatible with the inclusion morphisms $\mathcal{V}^{(i)} \hookrightarrow \mathcal{V}$ and projection morphisms $\mathcal{V}^{(i)} \rightarrow \mathcal{V}_i$ in the sense of Definition \ref{n-rig def}. 

Finally, we say such a family is $\theta$-semistable if for each $s \in S$ the sheaf $\mathcal{V}_s$ is $\theta$-semistable.
\end{defn}

\begin{lemma}\label{univ fam}
There exists a flat family $\mathcal{V}$ of $n$-rigidified sheaves of Harder--Narasimhan type $\tau$ over $W$ parametrised by $Y_{(\tau)}^{ss}$ which is given by restricting the universal quotient sheaf $\mathcal{U}$ on $\mathrm{Quot}(V \otimes \mathcal{O}(-n),P) \times W$ to $Y_{(\tau)}^{ss} \times W$.
\end{lemma}
\begin{proof}
We use the vector space filtration $0 \subset V^{(1)} \subset \cdots \subset V^{(s)} = V$ corresponding to $\beta=\beta(\tau)$, defined as in $\S$7, to induce a universal Harder--Narasimhan filtration for $\mathcal{V}$. Then a universal $n$-rigidification comes from the eigenspace decomposition $V = \oplus_{i=1}^s V_i$ for $\beta$.
\end{proof}

\begin{defn} The moduli functor of $\theta$-semistable $n$-rigidified sheaves over $W$ of Harder--Narasimhan type $\tau$ is the contravariant functor $\mathcal{M}^{\theta-ss}(W, \tau, n)$ from complex schemes to sets such that
if $S$ is a scheme over $\mathbb{C}$ then
$\mathcal{M}^{\theta-ss}(W, \tau, n) (S)$ is the set of isomorphism classes of families of $\theta$-semistable $n$-rigidified sheaves over $W$ parametrised by $S$ with Harder--Narasimhan type $\tau$.
\end{defn}

\subsection{Boundedness}

By Lemma \ref{HN bdd} if $n$ is sufficiently large then all sheaves with Hilbert polynomial $P$ and Harder--Narasimhan type $\tau$ are $n$-regular and  the successive quotients appearing in their Harder--Narasimhan filtrations are $n$-regular. A similar argument gives us

\begin{lemma}\label{subsheaves bdd}
Fix a Harder--Narasimhan type $\tau = (P_1, \dots, P_s)$. Then for $n$ sufficiently large every $\tau$-compatible subsheaf $\mathcal{F}' \subset \mathcal{F}$ of a sheaf with Harder--Narasimhan type $\tau$ is $n$-regular. Moreover, the successive quotients $\mathcal{F}'_i$ appearing in the generalised Harder--Narasimhan filtration of $\mathcal{F}'$ are also $n$-regular. 
\end{lemma}

We also have

\begin{lemma}\label{inequality for n}
Fix a Harder--Narasimhan type $\tau = (P_1, \dots, P_s)$. If $n$ is
sufficiently large, then for any $\tau$-compatible subsheaf $\mathcal{F}' \subset \mathcal{F}$ of a sheaf with Harder--Narasimhan type $\tau$ the following inequalities are equivalent:
\[ \frac{\sum \theta_i P(\mathcal{F}'_i)}{P(\mathcal{F}')} \geq \frac{\sum \theta_i P(\mathcal{F}_i)}{P(\mathcal{F})} \iff \frac{\sum \theta_i P(\mathcal{F}'_i,n)}{P(\mathcal{F}',n)} \geq \frac{\sum \theta_i P(\mathcal{F}_i,n)}{P(\mathcal{F},n)} \]
where  $\mathcal{F}'_i$ and $\mathcal{F}_i$ are the successive quotients
 in the generalised Harder--Narasimhan filtrations of  $\mathcal{F}'$
 and $\mathcal{F}$. 
\end{lemma}
\begin{proof}
The Hilbert polynomials of $\mathcal{F}$ and $\mathcal{F}_i$ are fixed, and 
the successive quotients $\mathcal{F}_i'$ are semistable with reduced Hilbert polynomial
\[ \frac{P(\mathcal{F}_i')}{r_i'} = \frac{P_i}{r_i}\]
where $r_i'$ denotes the multiplicity of $\mathcal{F}_i'$, so since there are only a finite number of possibilities for $r_i'$, there are only a finite number of possible Hilbert polynomials for $\mathcal{F}_i'$.
Thus the inequalities are equivalent for all sufficiently large $n$.
\end{proof}

\subsection{The choice of perturbed linearisation}\label{choice of pert linear}
Let $\theta = (\theta_1, \ldots,  \theta_s) \in \mathbb{Q}^s$ be
a stability parameter satisfying the condition (\ref{cross})
of Remark \ref{assum on theta}. Then
$\theta$ defines a perturbation of the canonical linearisation $\mathcal{L}_\beta$ in the following way. For any natural number $n$ we can define
\begin{equation} \label{beta'def} \beta_i' := \theta_i -\frac{\sum_{j=1}^s \theta_j P_j(n)}{P(n)} \end{equation}
and let $\beta' :=  i \mathrm{diag}(\beta_1', \cdots \beta_1', \cdots , \beta_s', \cdots \beta_s') \in \mathfrak{t}$
where $\beta_i'$ appears $P_i(n)$ times. Then \[\sum_{i=1}^s \beta'_i P_i(n) = 0\]
  and the assumption (\ref{cross}) on $\theta$ means that \[\beta' \cdot \beta =\sum_{i=1}^s \beta'_i \beta_i P_i(n) \geq 0.\]
For any small positive rational number $\epsilon$ consider the perturbation $\mathcal{L}_\beta^{\mathrm{per}}$ of the canonical $\mathrm{Stab} \beta$-linearisation $\mathcal{L}_\beta$ on $\overline{Y}_{(\tau)}$ given by twisting the original ample linearisation $\mathcal{L}$ on $\overline{Q}$ by the character $\chi_{-(\beta + \epsilon \beta')} : \mathrm{Stab} \beta \rightarrow \mathbb{C}^*$ corresponding to the rational weight $-(\beta + \epsilon \beta')$.
By Proposition \ref{small per refines strat} if $\epsilon >0 $ is sufficiently small then the stratification associated to the $\mathrm{Stab} \beta $-action on $\overline{Y}_{(\tau)}$ with respect to $\mathcal{L}_\beta^{\mathrm{per}}$ is a refinement of the stratification associated to the $\mathrm{Stab} \beta $-action on $\overline{Y}_{(\tau)}$ with respect to $\mathcal{L}_\beta$. We assume that $\epsilon > 0$ is sufficiently small for
this to be the case, and then since $\overline{Y}_{(\tau)}^{\mathrm{Stab} \beta -ss}(\mathcal{L}_\beta) =Y_{(\tau)}^{ss}$ it follows that 
\[{Y_{(\tau)}^{ss}}= \bigsqcup_{\gamma \in \mathcal{C}} S_\gamma^{(\beta)} \]
where $S_\gamma^{(\beta)}$ is a stratum appearing in the stratification for the perturbed linearisation,
and we have for each $\gamma \in \mathcal{C}$
\[ S_\gamma^{(\beta)} = G  Y_\gamma^{(\beta)-ss} \cong G \times_{P_\beta}
Y_\gamma^{(\beta)-ss}
\]
where $Y_\gamma^{(\beta)-ss} = (p_\gamma^{(\beta)})^{-1}(Z_\gamma^{(\beta)-ss})$
and $Y_\gamma^{(\beta)-ss}$ and $Z_\gamma^{(\beta)-ss}$ are the subschemes of $Y_{(\tau)}^{ss}$ defined following (\ref{zedbeta})
and (\ref{ybeta}).

A 1-PS $\lambda:\mathbb{C}^* \to \mathrm{Stab} \beta
\cong \mathrm{SL}(V) \cap \Pi \mathrm{GL}(V_i)$ of $\mathrm{Stab} \beta$ is given by 1-PSs $\lambda_i: \mathbb{C} \to \mathrm{GL}(V_i)$ for $i=1, \dots, s$ such that
\[ \prod_{i=1}^s \mathrm{det} \lambda_i(t) = 1 \]
for all $t \in \mathbb{C}^*$.
As in $\S$\ref{calc HM fn} we can diagonalise each 1-PS 
simultaneously to get weights $k_1 > \cdots > k_r$ and for each $i$ 
 a decomposition $V_i =V_i^1 \oplus \cdots \oplus V_i^r$ into weight spaces and 
a filtration
\[ 0 \subset V_i^{[1]} \subset \cdots \subset V_i^{[r]} = V_i\]
of $V_i$ where $V_i^{[j]}:= \oplus_{l \leq j} V_i^l$ such that
\[ \sum_{i=1}^s \sum_{j=1}^r k_j \mathrm{dim} V_i^j = 0.\]
There is an associated filtration 
\[ 0 \subset V^{[1]} \subset \cdots \subset V^{[r]}= V\]
of $V$ where $V^{[j]}:= \oplus_{i=1}^s V_i^{[j]}$ and we let $V^j:= V^{[j]}/V^{[j-1]}$.

Now suppose $\rho : V \otimes \mathcal{O}(-n) \to \mathcal{F}$ is a point in $Y_{(\tau)}^{ss}$ such that the limit $\overline{\rho} := \lim_{t \to 0} \lambda(t) \cdot \rho $ is also in $Y_{(\tau)}^{ss}$. Then the 1-PS $\lambda$ determines a filtration
\[ 0 \subset \mathcal{F}^{[1]} \subset \cdots \subset \mathcal{F}^{[r]}= \mathcal{F} \]
where $ H^0(\mathcal{F}^{[j]}(n)) =V^{[j]}$ and $\overline{\rho} = \oplus_{j=1}^r \rho^j$ where $\overline{\rho}^j : V^j \otimes \mathcal{O}(-n) \rightarrow \mathcal{F}^j:= \mathcal{F}^{[j]}/ \mathcal{F}^{[j-1]}$. As $\overline{\rho}$ is also a point in $Y_{(\tau)}^{ss}$, the sheaf $\overline{\mathcal{F}}:= \oplus_{j=1}^r \mathcal{F}^j$ has Harder--Narasimhan type $\tau$ and the filtration $0 \subset V^{(1)} \subset \cdots \subset V^{(s)} = V$ induces this filtration. In particular each direct summand $\mathcal{F}^j$ is $\tau$-compatible (see Lemma \ref{HN direct sum}) and has generalised Harder--Narasimhan filtration
\[  0 \subseteq \mathcal{F}^j_{(1)} \subseteq \cdots \subseteq   \mathcal{F}^j_{(s)} = \mathcal{F}^j .\]
We let $\mathcal{F}_i^j$ denote the successive quotients in this generalised Harder--Narasimhan filtration. 

\begin{lemma}\label{helpful lemma}
Suppose $m >\!> n>\!> 0$ and let $\lambda$ be a 1-PS of $\mathrm{Stab} \beta$ and $\rho : V \otimes \mathcal{O}(-n) \to \mathcal{F}$ be a point in $Y_{(\tau)}^{ss}$. If the limit $\overline{\rho} := \lim_{t \to 0} \lambda(t) \cdot \rho $ is also in $Y_{(\tau)}^{ss}$ then, using the above notation, we have
\begin{enumerate}
\renewcommand{\labelenumi}{\textrm{\roman{enumi})}} 
\item for $0 \leq l < j \leq r$ the quotient sheaf $ \mathcal{F}^{[j]}/\mathcal{F}^{[l]}$ is $\tau$-compatible with generalised Harder--Narasimhan filtration induced by that of $\mathcal{F}$;
\item the Hilbert--Mumford function is given by
\[ \mu^{\mathcal{L}_\beta^{\mathrm{per}}}(\rho, \lambda) = \epsilon \sum_{j=1}^r \sum_{i=1}^s  k_j   \beta_i'P(\mathcal{F}_i^j,n) . \]
\end{enumerate}
\end{lemma}
\begin{proof}
Let $n >\!> 0$ so that Lemma \ref{subsheaves bdd} holds. Let $0 \subset \overline{\mathcal{F}}^{(1)} \subset \cdots \subset \overline{\mathcal{F}}^{(s)}= \overline{\mathcal{F}}$ denote the Harder--Narasimhan filtration of $\overline{\mathcal{F}}$ where $V^{(i)} \cong H^0(\overline{\mathcal{F}}^{(i)}(n))$. By Lemma \ref{HN direct sum} the direct summands $\mathcal{F}^j$ have generalised Harder--Narasimhan filtrations  
\[  0 \subset  \mathcal{F}^j_{(1)} \subset \cdots \subset  
\mathcal{F}^j_{(s)} = \mathcal{F}^{j}\]
where
\[ \mathcal{F}^j_{(i)} := \mathcal{F}^j \cap \overline{\mathcal{F}}^{(i)}  =\overline{\rho}({V}^{(i)} \otimes \mathcal{O}(-n) ) \cap \overline{\rho}(V^j\otimes \mathcal{O}(-n) ). \]
Since $\overline{\rho}$ is a direct sum of maps which send $V^j\otimes \mathcal{O}(-n) $ to $\mathcal{F}^j$ and $H^0(\overline{\rho}(n))$ is an isomorphism this is equal to
 \[ \mathcal{F}^j_{(i)}=\overline{\rho}(({V}^{(i)}  \cap V^j)\otimes \mathcal{O}(-n) )= \frac{\rho((V^{(i)} \cap V^{[j]}) \otimes \mathcal{O}(-n)  )}{\rho((V^{(i)} \cap V^{[j-1]}) \otimes \mathcal{O}(-n) )} . \]
Let $\mathcal{F}^{[j]}_{(i)}:= \rho((V^{(i)} \cap V^{[j]}) \otimes \mathcal{O}(-n)  )$; then these sheaves define a filtration
\[ 0 \subset \mathcal{F}^{[j]}_{(1)} \subset \cdots \subset \mathcal{F}^{[j]}_{(s)} = \mathcal{F}^{[j]}\]
of $\mathcal{F}^{[j]}$. We claim that this filtration is a generalised Harder--Narasimhan filtration for $\mathcal{F}^{[j]}$ and thus that $\mathcal{F}^{[j]}$ is $\tau$-compatible. It is enough to show that $\mathcal{F}^{[j]}_{i}:= \mathcal{F}_{(i)}^{[j]}/\mathcal{F}_{(i-1)}^{[j]}$ is semistable with reduced Hilbert polynomial $P_i/r_i$ if it is nonzero. We prove this by induction on $j$. For $j=1$ it is clear as $\mathcal{F}^{[1]}$ is $\tau$-compatible so suppose we know this is true for $j-1$. We have a diagram of short exact sequences
\[  \xymatrix{
& 0   \ar[d]  & 0 \ar[d] \\ 0 \ar[r] & \mathcal{F}_{(i-1)}^{[j-1]}\ar[r] \ar[d] & \mathcal{F}_{(i)}^{[j-1]} \ar[r] \ar[d] & \mathcal{F}_{i}^{[j-1]} \ar[r] \ar[d] & 0 \\
0 \ar[r] & \mathcal{F}_{(i-1)}^{[j]} \ar[r] \ar[d] & \mathcal{F}_{(i)}^{[j]} \ar[r] \ar[d] & \mathcal{F}_{i}^{[j]} \ar[r] \ar[d] & 0 \\
0 \ar[r] & \mathcal{F}_{(i-1)}^j \ar[r] \ar[d] & \mathcal{F}_{(i)}^j \ar[r] \ar[d] & \mathcal{F}_{i}^j \ar[r]  & 0  \\
& 0 & 0 & \\
}\]
and so $\mathcal{F}_{(i-1)}^{[j-1]} =  \mathcal{F}_{(i)}^{[j-1]} \cap  \mathcal{F}_{(i-1)}^{[j]}$ from which it follows that the third column is also a short exact sequence. As the outer sheaves in this short exact sequence are both semistable with reduced Hilbert polynomial $P_i/r_i$, so is the middle sheaf. 
This completes the induction and shows that $ \mathcal{F}^{[j]}/\mathcal{F}^{[l]}$ is also $\tau$-compatible.

Recall that $\mathcal{L}_\beta^{\mathrm{per}}$ was constructed by twisting the original linearisation $\mathcal{L}$ on $\overline{Y}_{(\tau)}$ by the character of $\mathrm{Stab} \beta$ corresponding to $-(\beta + \epsilon \beta')$; therefore,
\[ \mu^{\mathcal{L}_\beta^{\mathrm{per}}}(\rho, \lambda)  =\mu^{\mathcal{L}}(\rho, \lambda) +(\beta + \epsilon \beta') \cdot \lambda \]
where $\cdot$ denotes the natural pairing between characters and 1-PSs of $\mathrm{Stab} \beta$. We have calculated
\[ \mu^{\mathcal{L}}(\rho, \lambda) = \sum_{j=1}^r k_j P(\mathcal{F}^j,m) \]
(see Lemma \ref{hmf}) and
\[ (\beta + \epsilon \beta') \cdot \lambda = \sum_{i=1}^s \sum_{j=1}^r k_j (\beta_i + \epsilon \beta'_i)v_{i,j} \]
where $v_{i,j}$ is the dimension of $(V^j \cap V^{(i)}/V^j \cap V^{(i-1)})$. Observe that $v_{i,j}=P(\mathcal{F}_i^j,n)$ where $\mathcal{F}_i^j = \mathcal{F}_{(i)}^j/\mathcal{F}_{(i-1)}^j$ as $H^0(\overline{\rho}(n))$ is an isomorphism, so that $V^{j} \cap V^{(i)} \cong H^0(\mathcal{F}_{(i)}^{j}(n))$ and the $\mathcal{F}_i^j$ are all $n$-regular. Then since $\mathcal{F}^j$ is $\tau$-compatible this means $\mathcal{F}_i^j$, if nonzero, has reduced Hilbert polynomial equal to $P_i/r_i$ so
\[ \sum_{i=1}^s \frac{P_i(m)}{P_i(n)} v_{i,j} = \sum_{i=1}^s P(\mathcal{F}_i^j,m) = P(\mathcal{F}^j,m). \]
Thus
\[ \mu^{\mathcal{L}_\beta^{\mathrm{per}}}(\rho, \lambda) = \sum_{j=1}^r k_j \left( P(\mathcal{F}^{j},m) + \sum_{i=1}^s \left( \epsilon \beta_i'-\frac{P_i(m)}{P_i(n)} \right)v_{i,j} \right) = \epsilon \sum_{j=1}^r \sum_{i=1}^s  k_j   \beta_i'P(\mathcal{F}_i^j,n) \]
and the proof is complete.
\end{proof}

We can use this lemma to study the indices $\gamma \in \mathfrak{t}_+$ of the stratification
$\{S^{(\beta)}_\gamma: \gamma \in \mathcal{C}\}$ of $Y_{(\tau)}^{ss}$. Recall that $\gamma$ determines a 1-PS
$\lambda_{\gamma}$ of $\mathrm{Stab} \beta$, and as above this determines a decomposition 
$V = V^1 \oplus \cdots \oplus V^r$
 of $V$ into weight spaces and an associated filtration
$ 0 \subset {V}^{[1]} \subset \cdots \subset V^{[r]}=V$ 
where $V^{[j]} = \bigoplus_{l \leq j} V^l$, together with a sequence of rational numbers $\gamma_1 > \cdots > \gamma_r$ such that $\sum \gamma_j \mathrm{dim} V^j = 0$.

\begin{prop}\label{second index set}
Suppose that $m >\!> n >\!> 0$ and that  
$\gamma$ is a nonzero index in the stratification
$\{S^{(\beta)}_\gamma: \gamma \in \mathcal{C}\}$ of $Y_{(\tau)}^{ss}$. 
If $\rho  : V \otimes \mathcal{O}(-n) \rightarrow \mathcal{F}$ belongs to
the subscheme
$Y_\gamma^{(\beta)-ss}$ of $Y_{(\tau)}^{ss}$, then 
$\overline{\rho} = p_\gamma^{(\beta)}(\rho)
\in Z_\gamma^{(\beta)-ss}$ is given by
$\overline{\rho} = \oplus_{j=1}^r \rho^j:\oplus_{j=1}^r V^j \otimes \mathcal{O}(-n) \to \oplus_{j=1}^r \mathcal{F}^j$ where
 $\mathcal{F}^{[j]} =\rho(V^{[j]} \otimes \mathcal{O}(-n))$ and $\mathcal{F}^j =\mathcal{F}^{[j]}/ \mathcal{F}^{[j-1]}$. In particular the $\mathcal{F}^j$ are $\tau$-compatible and so have generalised Harder--Narasimhan filtrations
 \[  0 \subseteq \mathcal{F}^j_{(1)} \subseteq \cdots \subseteq   \mathcal{F}^j_{(s)} = \mathcal{F}^j. \]
 Let $\mathcal{F}_i^j:=\mathcal{F}^j_{(i)}/\mathcal{F}^j_{(i-1)}$; then
\[ \gamma_j = - \frac{ \epsilon\sum_{i=1}^s \beta'_i P(\mathcal{F}_i^j,n) }{P(\mathcal{F}^j,n) }. \]
\end{prop}
\begin{proof} We assume $m >\!> n >\!> 0$ so that the statements of Proposition \ref{strats agree} and Lemma \ref{subsheaves bdd} hold. We have seen that 
\[\overline{\rho} = \oplus_{j=1}^r \rho^j:\oplus_{j=1}^r V^{j} \otimes \mathcal{O}(-n) \to \oplus_{j=1}^r \mathcal{F}^j \]
 is the graded object associated to the filtration $0= \mathcal{F}^{[0]} \subset \mathcal{F}^{[1]} \subset \cdots \subset \mathcal{F}^{[r]} = \mathcal{F}$ of $\mathcal{F}$ given by $\mathcal{F}^{[j]} =\rho(V^{[j]} \otimes \mathcal{O}(-n))$,
by \cite{huybrechts} Lemma 4.4.3. In particular $\overline{\rho} \in Y_{(\tau)}^{ss}$, so by Lemma \ref{helpful lemma} the $\mathcal{F}^j$ are $\tau$-compatible and 
\[ \mu^{\mathcal{L}_\beta^{\mathrm{per}}}(\rho, \lambda_{\gamma}) = \epsilon \sum_{j=1}^r \sum_{i=1}^s  \gamma_j   \beta_i'P(\mathcal{F}_i^j,n) .\]
Since $\rho \in Y_\gamma^{(\beta)-ss}$ the associated 1-PS $\lambda_{\gamma}$ is adapted to $\rho$, and so 
\[ \frac{\mu^{\mathcal{L}_\beta^{\mathrm{per}}}(\rho, \lambda)}{||\lambda||} \]
takes its minimum value for $\lambda$ a non-trivial 1-PS of $\mathrm{Stab} \beta$ when
$\lambda = \lambda_{\gamma}$; this will enable us to determine the values of $\gamma_j$ for $1 \leq j \leq r$. If we minimise the quantity
\[ \frac{ \mu^{\mathcal{L}_\beta^{\mathrm{per}}}(\rho, \lambda_{\gamma})}{|| \lambda_{\gamma} ||} \]
subject to $\sum_{i=1}^r \gamma_j P(\mathcal{F}^j,n) = 0$, we see that
\[ \gamma_j = - \frac{ \epsilon\sum_{i=1}^s \beta'_i P(\mathcal{F}_i^j,n) }{P(\mathcal{F}^j,n)) } \]
The $\gamma_j$ have been scaled so that $ \mu^{\mathcal{L}_\beta^{\mathrm{per}}}(\rho, \lambda_{\gamma}) = - || \gamma ||^2$, which ensures $\overline{\rho}$ is a point in $ Z_\gamma^{(\beta)}$.
\end{proof}

From this description we can write down the strata inductively, starting with the highest stratum. In particular we know the GIT semistable set, corresponding to the open stratum $S_0^{(\beta)}$, is the complement of the (closures of the) higher strata.

\begin{prop} \label{ss can be tested by subobjects}
Suppose $m >\!> n >\!> 0$ and $\rho: V \otimes \mathcal{O}(-n) \rightarrow \mathcal{F}$ is a point in $Y_{(\tau)}^{ss}$. Then $\rho$ is semistable with respect to $\mathcal{L}_\beta^{\mathrm{per}}$ if and only if for all proper nonzero $\tau$-compatible subsheaves $\mathcal{F}'  \subset \mathcal{F}$ for which $\mathcal{F} / \mathcal{F}'$ is $\tau$-compatible we have
\[\sum_{i=1}^s  \beta_i' P(\mathcal{F}'_i,n)  \geq 0\]
where since $\mathcal{F}'$ is $\tau$-compatible it has a generalised Harder--Narasimhan filtration
\[  0 \subseteq \mathcal{F}'_{(1)} \subseteq \cdots \subseteq   \mathcal{F}'_{(s)} = \mathcal{F}' \]
and $\mathcal{F}'_i:=\mathcal{F}'_{(i)}/\mathcal{F}'_{(i-1)}$.
\end{prop}
\begin{proof}
We suppose $m >\!> n >\!> 0$ are chosen as at the beginning of Proposition \ref{second index set}. Suppose $\rho$ is semistable with respect to $\mathcal{L}_\beta^{\mathrm{per}}$ and let $\mathcal{F}'  \subset \mathcal{F}$ be a $\tau$-compatible subsheaf such that $\mathcal{F} / \mathcal{F}'$ is $\tau$-compatible. Let $V'=H^0(\rho(n))^{-1}(H^0(\mathcal{F}'(n))) \subset V$ and let $V''$ be a complement to $V'$ in $V$. Consider the 1-PS
\[ \lambda(t) = \left(\begin{array}{cc} t^{v-v'}I_{V'} & 0 \\ 0 & t^{-v'}I_{V''} \end{array} \right)\]
where $v'$ (respectively $v$) denotes the dimension of $V'$ (respectively $V$). Then \[\overline{\rho} :=( \lim_{t \to 0} \lambda(t) \cdot \rho ): (V' \oplus V'') \otimes \mathcal{O}(-n) \rightarrow \overline{\mathcal{F}}\] where $ \overline{\mathcal{F}}= \mathcal{F}' \oplus \mathcal{F}/\mathcal{F}'$ has Harder--Narasimhan type $\tau$. Since $\rho$ is semistable
\[ \mu^{\mathcal{L}_\beta^{\mathrm{per}}}(\rho, \lambda) \geq 0, \]
but by Lemma \ref{helpful lemma} 
\[ \mu^{\mathcal{L}_\beta^{\mathrm{per}}}(\rho, \lambda) = v \epsilon  \sum_{i=1}^s  \beta_i' P(\mathcal{F}'_i,n) \]
where $v \epsilon > 0$, so $\sum_{i=1}^s  \beta_i' P(\mathcal{F}'_i,n) \geq 0$.

Now suppose $\rho$ is unstable with respect to $\mathcal{L}_\beta^{\mathrm{per}}$. Then there is a nonzero $\gamma \in \mathcal{C}$ such that $\rho$ belongs to $S^{(\beta)}_\gamma$, and in fact by conjugating $\gamma$ by an element of $\mathrm{Stab} \beta$ we may assume $\rho \in Y_\gamma^{(\beta)-ss}$. Then $\gamma$ determines a filtration $0 \subset V^{[1]} \subset \cdots \subset V^{[r]} = V$ and sequence of rational numbers $\gamma_1 > \cdots > \gamma_r$, and by Proposition \ref{second index set}
\[ \gamma_j = -\frac{\epsilon \sum_{i=1}^s \beta_i'P(\mathcal{F}^{j}_i,n)}{P(\mathcal{F}^j,n)}.\]
We claim for $j=2, \dots , r$ that
\[  \frac{ \sum_{i=1}^s \beta_i'P(\mathcal{F}^{[1]}_i,n)}{P(\mathcal{F}^{[1]},n)} < \frac{ \sum_{i=1}^s \beta_i'P(\mathcal{F}^{[j]}_i,n)}{P(\mathcal{F}^{[j]},n)}.\]
For $j=2$ this is equivalent to the inequality $\gamma_1 > \gamma_2$. Then we proceed by induction as combining the above inequality with $\gamma_1 > \gamma_{j+1}$ gives the inequality for $j+1$. In particular if $j=r$ then
\[  \frac{ \sum_{i=1}^s \beta_i'P(\mathcal{F}^{[1]}_i,n)}{P(\mathcal{F}^{[1]},n)} < \frac{ \sum_{i=1}^s \beta_i'P(\mathcal{F}_i,n)}{P(\mathcal{F},n)} =0\]
by construction of $\beta'$. Let $\mathcal{F}'= \mathcal{F}^{[1]}$; then by Lemma \ref{helpful lemma} both $\mathcal{F}'$ and $\mathcal{F}/ \mathcal{F}'$ are $\tau$-compatible and we have shown that
\[  \sum_{i=1}^s \beta_i'P(\mathcal{F}'_i,n) < 0.\]
\end{proof}

\subsection{Moduli of $\theta$-semistable $n$-rigidified sheaves of fixed Harder--Narasimhan type}

As before let $W$ be a complex projective scheme and let $\tau = (P_1, \dots, P_s)$ be a fixed Harder--Narasimhan type. Let $P= \sum_{i=1}^s P_i$ and for $n >\!>0$ let $V$ be a vector space of dimension $P(n)$. Recall that $Q$ is the open subscheme of $\mathrm{Quot}(V \otimes \mathcal{O}(-n),P)$ representing quotient sheaves $\rho : V \otimes \mathcal{O}(-n) \rightarrow \mathcal{F}$ which are pure of dimension $e$ and such that $H^0(\rho(n))$ is an isomorphism. We defined in $\S$\ref{refined strat} a subscheme $R_\tau = GY_{(\tau)}^{ss}$ of $Q$ consisting of the quotient sheaves $\rho : V \otimes \mathcal{O}(-n) \rightarrow \mathcal{F}$ which have Harder--Narasimhan type $\tau$. Let $\beta = \beta(\tau)$ be the corresponding index of the stratification $\{ S_\beta \: : \: \beta \in \mathcal{B} \}$ of $\overline{Q}$ as defined in $\S$\ref{strat of Q}, and recall from Proposition \ref{strats agree} that for $m >\!> n >\!> 0$ the subscheme $R_\tau$ is a union of connected components of $S_\beta$. A choice of $\theta \in \mathbb{Q}^s$ defines a notion of (semi)stability for sheaves of Harder--Narasimhan type $\tau$ (see Definition \ref{defthetastability}) and an ample $\mathrm{Stab} \beta$-linearisation $\mathcal{L}_\beta^{\mathrm{per}}$ on a projective completion $\overline{Y}_{(\tau)}$ of $Y_{(\tau)}^{ss}$ in terms of  \[\beta' = i \mathrm{diag}(\beta_1', \cdots \beta_1', \cdots , \beta_s', \cdots \beta_s') \in \mathfrak{t}\]
defined as at (\ref{beta'def}) where
\[ \beta_i' = \theta_i -\frac{\sum_{j=1}^s \theta_j P_j(n)}{P(n)} \]
appears $P_i(n)$ times (see $\S$\ref{choice of pert linear}). 

\begin{thm}\label{ss agrees}
Suppose $n$ is sufficiently large and for fixed $n$ that $m$ is sufficiently large. Then $\rho: V \otimes \mathcal{O}(-n) \rightarrow \mathcal{F}$ in $Y_{(\tau)}^{ss}$ is GIT semistable for the action of $\mathrm{Stab} \beta$
on $\overline{Y_{(\tau)}}$ with respect to $\mathcal{L}_\beta^{\mathrm{per}}$ if and only if $\mathcal{F}$ is $\theta$-semistable.
\end{thm}
\begin{proof}
We pick $n$ sufficiently large so that the statements of Lemma \ref{subsheaves bdd} and Lemma \ref{inequality for n} hold. Then pick $m$ as in \cite{simpson} so that GIT semistability of points in $Q$ with respect to $\mathcal{L}$ is equivalent to semistability of the corresponding sheaf. We also assume $n$ and $m$ are chosen large enough for Proposition \ref{strats agree} to hold.

Suppose $\mathcal{F}$ is $\theta$-semistable and consider a $\tau$-compatible subsheaf $\mathcal{F}'$ of $\mathcal{F}$ such that the quotient $\mathcal{F}/ \mathcal{F}'$ is also $\tau$-compatible. Then by $\theta$-semistability we have an inequality
\[\frac{\sum \theta_i P(\mathcal{F}'_i,n)}{P(\mathcal{F}',n)} \geq \frac{\sum \theta_i P(\mathcal{F}_i,n)}{P(\mathcal{F},n)} \]
which by the definition of $\beta'$ is equivalent to $ \sum \beta_i' P(\mathcal{F}'_i ,n) \geq 0$, and so by 
Proposition \ref{ss can be tested by subobjects} we conclude that $\rho$ is GIT semistable with respect to $\mathcal{L}_\beta^{\mathrm{per}}$.

Now suppose $\rho$ is GIT semistable with respect to $\mathcal{L}_\beta^{\mathrm{per}}$ and take a $\tau$-compatible subsheaf $\mathcal{F}' \subset \mathcal{F}$ such that $\mathcal{F}/ \mathcal{F}'$ is  $\tau$-compatible. Then $\sum \beta_i' P(\mathcal{F}'_i ,n) \geq 0$ by Proposition  \ref{ss can be tested by subobjects}, or equivalently 
\[\frac{\sum \theta_i P(\mathcal{F}'_i,n)}{P(\mathcal{F}',n)} \geq \frac{\sum \theta_i P(\mathcal{F}_i,n)}{P(\mathcal{F},n)} .\]
We have chosen $n$ so that we can apply the results of Lemma \ref{inequality for n} and conclude that $\mathcal{F}$ is $\theta$-semistable.
\end{proof}

\begin{rmk}
It is straightforward to modify the proof of this theorem to show that, under the same assumptions on $n$ and $m$, a point $\rho: V \otimes \mathcal{O}(-n) \rightarrow \mathcal{F} $ in $Y_{(\tau)}^{ss}$ is GIT stable 
 with respect to $\mathcal{L}_\beta^{\mathrm{per}}$ 
if and only if the sheaf $\mathcal{F}$ of Harder--Narasimhan type $\tau$ is $\theta$-stable in the sense of Definition \ref{defthetastability}. 
\end{rmk}

\begin{rmk}\label{ss agrees rmk}
Our aim is to take a GIT quotient of $\overline{Y}_{(\tau)}$ by the action of
$\mathrm{Stab} \beta$,  so we need to examine semistability here. If  a point
$\rho: V \otimes \mathcal{O}(-n) \rightarrow \mathcal{F}$ 
in $\overline{Y}_{(\tau)}$ is $\theta$-semistable then the sheaf
$\mathcal{F}$ is $\tau$-compatible, and since $\mathcal{F}$ also has Hilbert polynomial $P$ it must have Harder--Narasimhan type $\tau$, so that $\rho$ actually belongs to $Y_{(\tau)}^{ss}$. Let
\[ Y_{(\tau)}^{\theta -ss}:= \overline{Y}_{(\tau)}^{\theta-ss}\]
be the set of $\theta$-semistable sheaves in $\overline{Y}_{(\tau)}$; then as we saw above this set is contained in ${Y}_{(\tau)}^{ss}$. We are assuming that
$\epsilon>0$ is sufficiently small that the perturbation $\mathcal{L}_\beta^{\mathrm{per}}$ of $\mathcal{L}_\beta$ satisfies Proposition \ref{small per refines strat}. Therefore it follows from Theorem \ref{ss agrees} 
that on $\overline{Y}_{(\tau)}$ GIT (semi)stability 
 with respect to $\mathcal{L}_\beta^{\mathrm{per}}$ of a point 
$\rho: V \otimes \mathcal{O}(-n) \rightarrow \mathcal{F}$ 
is equivalent to $\theta$-(semi)stability 
of the quotient sheaf $\mathcal{F}$ for $n$ and $m$ sufficiently large.  
\end{rmk}

\begin{defn}
Let $\mathcal{F}$ be a $\theta$-semistable $n$-rigidified sheaf of Harder--Narasimhan type $\tau$. A Jordan--H\"{o}lder filtration of $\mathcal{F}$ with respect to $\theta$ is a filtration
\[ 0 \subset \mathcal{F}^{\{ 1 \}} \subset \cdots \subset \mathcal{F}^{\{ r \}}= \mathcal{F} \]
such that:
\begin{enumerate} \item The successive quotients $\mathcal{F}^j:=\mathcal{F}^{\{j\}}/ \mathcal{F}^{\{j-1\}}$ are $\tau$-compatible and $\theta$-stable with
\[\frac{\sum_{i=1}^s \theta_i P(\mathcal{F}^j_i)}{P( \mathcal{F}^j)} = \frac{\sum \theta_i P(\mathcal{F}_i)}{P(\mathcal{F})}. \]
\item The $n$-rigidification for $\mathcal{F}$ induces generalised $n$-rigidifications for each $\mathcal{F}^j$; that is, an isomorphism $H^0(\mathcal{F}^j(n)) \cong \oplus_{i=1}^s H^0(\mathcal{F}^j_i(n))$ with the usual compatibilities.
\end{enumerate}
The associated graded sheaf $\oplus_{j=1}^r \mathcal{F}^j$ thus has an $n$-rigidification and is of Harder--Narasimhan type $\tau$. Moreover, this sheaf is $\theta$-polystable; i.e., a direct sum of $\theta$-stable sheaves. Standard arguments show that the $n$-rigidified sheaf $\oplus_{j=1}^r \mathcal{F}^j$ is uniquely determined up to isomorphism by $\mathcal{F}$.
Finally, we say two $\theta$-semistable $n$-rigidified sheaves $\mathcal{F}$ and $\mathcal{G}$ of Harder--Narasimhan type $\tau$ are S-equivalent if they have Jordan--H\"{o}lder filtrations such that the associated graded sheaves are isomorphic as $n$-rigidified sheaves. 
\end{defn}

\begin{rmk}\label{JH and S-equiv}
In exactly the same way as in the original proofs for S-equivalence of semistable sheaves, we see that the $\mathrm{Stab} \beta$-orbit closures in  $Y_{(\tau)}^{\theta -ss}$ of two $n$-rigidified sheaves $\mathcal{F}$ and $\mathcal{G}$ in $Y_{(\tau)}^{\theta -ss}$ intersect if and only if $\mathcal{F}$ and $\mathcal{G}$ are S-equivalent, and that the $\mathrm{Stab} \beta$-orbit of a point $\rho: V \otimes \mathcal{O}(-n) \rightarrow \mathcal{F}$ in $Y_{(\tau)}^{\theta -ss}$ is closed if and only if $\mathcal{F}$ is polystable. We briefly recap the argument here. From the general theory of GIT we know that the closure  in  $Y_{(\tau)}^{\theta -ss}$ of any $\mathrm{Stab} \beta$-orbit contains a unique closed $\mathrm{Stab} \beta$-orbit.
For any 
$\rho: V \otimes \mathcal{O}(-n) \rightarrow \mathcal{F}$ in $Y_{(\tau)}^{\theta -ss}$ we can choose a 1-PS whose limit as $t$ tends to zero is the graded object $\overline{\rho}: V \otimes \mathcal{O}(-n) \rightarrow \overline{\mathcal{F}}
$ associated to a Jordan--H\"{o}lder filtration of $\mathcal{F}$, so that $\overline{\mathcal{F}}
$ is polystable and $\overline{\rho}$ is in the orbit closure of $\rho$. 
Now suppose that 
$\rho: V \otimes \mathcal{O}(-n) \rightarrow \mathcal{F}$ is a point in $Y_{(\tau)}^{\theta -ss}$ such that $\mathcal{F}$ is a polystable sheaf $\mathcal{F} = \oplus \mathcal{F}_i$ and suppose that 
$\rho': V \otimes \mathcal{O}(-n) \rightarrow \mathcal{F'}$ in $Y_{(\tau)}^{\theta -ss}$ lies in the orbit closure of $\rho$. 
Then there is a family $\mathcal{V}$ of $\theta$-semistable sheaves parameterised 
by a curve $C$ such that $\mathcal{V}_{c_0} = \mathcal{F}'$ for some $c_0 \in C$ and for $c \neq c_o$ the corresponding sheaf is $\mathcal{V}_c = \mathcal{F}$. By semicontinuity
\[ \mathrm{hom}(\mathcal{F}_i, \mathcal{F}') \geq  \mathrm{hom}(\mathcal{F}_i, \mathcal{F}) \]
and by $\theta$-stability of $\mathcal{F}_i$ and $\theta$-semistability of $\mathcal{F}'$ we see that each nonzero morphism $\mathcal{F}_i \rightarrow \mathcal{F}'$ must be injective. 
From this we can conclude that $\mathcal{F}' \cong \oplus \mathcal{F}_i = \mathcal{F}$ and that $\rho'$ lies in the same $\mathrm{Stab} \beta$-orbit as $\rho$, so
the $\mathrm{Stab} \beta$-orbit of $\rho$ is closed. Thus
the unique closed $\mathrm{Stab} \beta$-orbit in the $\mathrm{Stab} \beta$-orbit closure in  $Y_{(\tau)}^{\theta -ss}$ of any point 
$\rho: V \otimes \mathcal{O}(-n) \rightarrow \mathcal{F}$
in $ Y_{(\tau)}^{\theta -ss}$ is the orbit of the graded object $\overline{\rho}: V \otimes \mathcal{O}(-n) \rightarrow \overline{\mathcal{F}}
$ associated to a Jordan--H\"{o}lder filtration of $\mathcal{F}$.
\end{rmk}

Just as for moduli of semistable sheaves over a projective scheme $W$
(cf. \cite{simpson} Theorem 1.21), we obtain
a projective scheme which corepresents the moduli functor of $\theta$-semistable $n$-rigidified sheaves of Harder--Narasimhan type $\tau$ over $W$,
in the sense of \cite{simpson} $\S$1 or \cite{ack} Definition 4.6.

\begin{thm} \label{thm6.16}
Let $W$ be a projective scheme over $\mathbb{C}$ and $\tau = (P_1, \dots, P_s)$ be a fixed Harder--Narasimhan type. For $\theta \in \mathbb{Q}^s$ and $n >\!> 0$ there is a projective scheme $M^{\theta-ss}(W, \tau, n)$ which corepresents the moduli functor $\mathcal{M}^{\theta-ss}(W, \tau, n)$ of $\theta$-semistable $n$-rigidified sheaves of Harder--Narasimhan type $\tau$ over $W$. The points of $M^{\theta-ss}(W, \tau,n)$ correspond to S-equivalence classes of $\theta$-semistable $n$-rigidified sheaves with Harder--Narasimhan type $\tau$.
\end{thm}
\begin{proof}
The proof is based on that of \cite{simpson} Theorem 1.21 (see also
\cite{ack} $\S$4).
Pick $n$ and $m$ as in the beginning of Theorem \ref{ss agrees}. For a complex scheme $R$ let $\underline{R}= \mathrm{Hom}( - , R)$ denote its functor of points, and if $R$ has a $G$-action then let $\underline{R} /\underline{G}$ denote the quotient functor. 

Let $\overline{Y}_{(\tau)}$ be the closure of ${Y}_{(\tau)}^{ss}$ as at the beginning of $\S$\ref{moduli unstable sheaves} and $\mathcal{L}_\beta^{\mathrm{per}}$ the linearisation defined in $\S$\ref{choice of pert linear}; then let
\[M^{\theta-ss}(W, \tau,n) := \overline{Y}_{(\tau)} /\!/_{\mathcal{L}_\beta^{\mathrm{per}}} \mathrm{Stab} \beta.\]
By Theorem \ref{ss agrees} and Remark \ref{ss agrees rmk} the projective scheme $M^{\theta-ss}(W, \tau,n) $ is a categorical quotient of the open subset $Y_{(\tau)}^{\theta-ss} \subseteq 
 \overline{Y}_{(\tau)}$ parameterising points $\rho : V \otimes \mathcal{O}(-n) \rightarrow \mathcal{F}$ of $\overline{Y}_{(\tau)}$ such that $\mathcal{F}$ is $\theta$-semistable for
the action of $\mathrm{Stab} \beta$, or equivalently by the action of $H:=\Pi_{i=1}^s \mathrm{GL}(V_i)$  since the central 1-PS $\mathbb{C}^* \subset \mathrm{GL}(V)$ acts trivially on $\overline{Y}_{(\tau)}$. 
The quotient map ${Y}_{(\tau)}^{\theta-ss} \rightarrow M^{\theta-ss}(W, \tau,n)$ is $H$-invariant and so induces a natural transformation 
\[\varphi_1 :  \underline{ Y_{(\tau)}}^{\theta-ss}/\underline{H} \rightarrow \underline{M^{\theta-ss}(W, \tau,n)},\]
and as $M^{\theta-ss}(W, \tau,n)$ is a categorical quotient it corepresents the quotient functor $\underline{ Y_{(\tau)}}^{\theta-ss}/\underline{H}$.

Let $\mathcal{V}$ denote the restriction to ${Y}_{(\tau)}^{\theta-ss}$ of the family of $\theta$-semistable $n$-rigidified sheaves of Harder--Narasimhan type $\tau$ parameterised by ${Y}_{(\tau)}^{ss}$ (cf. Lemma \ref{univ fam}). Then this family defines a natural transformation
\[ \phi : \underline{ Y_{(\tau)}}^{\theta-ss} \rightarrow \mathcal{M}^{\theta-ss}(X, \tau,n) \]
by sending a morphism $f : S \rightarrow Y_{(\tau)}^{\theta-ss}$ to the family $f^* \mathcal{V}$ for any scheme $S$. Following Lemma \ref{n-rig lemma} two elements of $ \underline{ Y_{(\tau)}}^{\theta-ss}(S)$ define isomorphic families if and only if locally on
$S$ they are related by an element of $\underline{H}(S)$, this descends to a local isomorphism (in the sense of \cite{simpson} $\S$1 or \cite{ack} Definition 4.3)
\[ \tilde{\phi} : \underline{ Y_{(\tau)}}^{\theta-ss}/ \underline{H} \rightarrow \mathcal{M}^{\theta-ss}(X, \tau, n).\]
Since local isomorphism means isomorphism after sheafification and $M^{\theta-ss}(W, \tau, n)$ corepresents $ \underline{ Y_{(\tau)}}^{\theta-ss}/\underline{H}$, it also corepresents $\mathcal{M}^{\theta-ss}(X, \tau, n)$ (cf. \cite{ack} Lemma 4.7). Finally the fact that the points of $M^{\theta-ss}(W, \tau, n)$ correspond to S-equivalence classes follows from Remark \ref{JH and S-equiv}.
\end{proof}

\bibliographystyle{amsplain}
\bibliography{refs_arxiv3}

\medskip \medskip

\noindent \textsc{Mathematical Institute, University of Oxford,}\\
\textsc{24-29 St. Giles', Oxford, OX1 3LB, UK.}
\medskip

\noindent \texttt{hoskins@maths.ox.ac.uk} \\  
 \texttt{kirwan@maths.ox.ac.uk}
\end{document}